\documentclass[11pt,english]{extarticle}
\usepackage[T3,T1]{fontenc}
\usepackage[latin9]{inputenc}
\usepackage{color}
\usepackage{babel}
\usepackage{mathtools}
\usepackage{amsmath}
\usepackage{amsthm}
\usepackage{amssymb}
\usepackage[unicode=true,pdfusetitle,
 bookmarks=true,bookmarksnumbered=true,bookmarksopen=true,bookmarksopenlevel=3,
 breaklinks=false,pdfborder={0 0 1},backref=false,colorlinks=true]
 {hyperref}
\hypersetup{
 pdfborderstyle=,pdfborderstyle={},pdfborderstyle={},pdfborderstyle={},pdfborderstyle={},pdfborderstyle={},pdfborderstyle={},pdfborderstyle={},pdfborderstyle={},pdfborderstyle={},pdfborderstyle={},pdfborderstyle={},pdfborderstyle={},pdfborderstyle={},pdfborderstyle={},pdfborderstyle={},pdfborderstyle={},pdfborderstyle={},pdfborderstyle={},pdfborderstyle={},pdfpagelayout=OneColumn,pdfnewwindow=true,pdfstartview=XYZ,plainpages=false,linkcolor=blue,urlcolor=blue,citecolor=red,anchorcolor=blue,linkcolor=blue,urlcolor=blue,citecolor=red,anchorcolor=blue}

\makeatletter
\theoremstyle{plain}
\newtheorem{thm}{\protect\theoremname}
\theoremstyle{plain}
\newtheorem{lem}{\protect\lemmaname}
\theoremstyle{definition}
\newtheorem{defn}{\protect\definitionname}
\theoremstyle{plain}
\newtheorem{fact}{\protect\factname}

\usepackage{babel}
\usepackage[mathscr]{eucal}
\usepackage{epsfig,epsf,psfrag}
\usepackage{amssymb,amsmath,amsthm,latexsym}
\usepackage{amsmath,graphicx,xcolor,url}
\usepackage[caption=false]{subfig} 
\usepackage{fixltx2e}%ordering of single and double column floats
\usepackage{array}%array and tabular environments
\usepackage{verbatim}
\usepackage{bm}
\usepackage{algorithmic}
\usepackage{algorithm}
\usepackage{verbatim}
\usepackage{textcomp}
\usepackage{mathrsfs,overpic}
\usepackage{epstopdf}
\usepackage{amsfonts}
\usepackage[numbers]{natbib}
\mathtoolsset{showonlyrefs} 

\usepackage{tikz}
\usepackage{float}
\usepackage{tabularx}
\usepackage{multirow}
\usetikzlibrary{patterns}

\newcommand{\Unif}{\mathrm{Unif}}

\def\esssup{{\mathrm{esssup}}}

\def\L{\mathsf{L}}
\def\P{{\rm P}}
\def\Q{{\rm Q}}
\def\R{{\rm R}}

\def\Levy{  d_{\rm P}}

\def\cconv{{\rm cconv}}
\def\conv{{\rm conv}}

\def\1{\mathbf{1}}

\def\d{{\text {\rm d}}}
\catcode`~=11 \def\UrlSpecials{\do\~{\kern -.15em\lower .7ex\hbox{~}\kern .04em}} \catcode`~=13 

\allowdisplaybreaks[1]

\DeclareMathAlphabet{\mathbsf}{OT1}{cmss}{bx}{n}
\DeclareMathAlphabet{\mathssf}{OT1}{cmss}{m}{sl}% slanted sans serif

\DeclareSymbolFont{bsfletters}{OT1}{cmss}{bx}{n}  
\DeclareSymbolFont{ssfletters}{OT1}{cmss}{m}{n}
\DeclareMathSymbol{\bsfGamma}{0}{bsfletters}{'000}
\DeclareMathSymbol{\ssfGamma}{0}{ssfletters}{'000}
\DeclareMathSymbol{\bsfDelta}{0}{bsfletters}{'001}
\DeclareMathSymbol{\ssfDelta}{0}{ssfletters}{'001}
\DeclareMathSymbol{\bsfTheta}{0}{bsfletters}{'002}
\DeclareMathSymbol{\ssfTheta}{0}{ssfletters}{'002}
\DeclareMathSymbol{\bsfLambda}{0}{bsfletters}{'003}
\DeclareMathSymbol{\ssfLambda}{0}{ssfletters}{'003}
\DeclareMathSymbol{\bsfXi}{0}{bsfletters}{'004}
\DeclareMathSymbol{\ssfXi}{0}{ssfletters}{'004}
\DeclareMathSymbol{\bsfPi}{0}{bsfletters}{'005}
\DeclareMathSymbol{\ssfPi}{0}{ssfletters}{'005}
\DeclareMathSymbol{\bsfSigma}{0}{bsfletters}{'006}
\DeclareMathSymbol{\ssfSigma}{0}{ssfletters}{'006}
\DeclareMathSymbol{\bsfUpsilon}{0}{bsfletters}{'007}
\DeclareMathSymbol{\ssfUpsilon}{0}{ssfletters}{'007}
\DeclareMathSymbol{\bsfPhi}{0}{bsfletters}{'010}
\DeclareMathSymbol{\ssfPhi}{0}{ssfletters}{'010}
\DeclareMathSymbol{\bsfPsi}{0}{bsfletters}{'011}
\DeclareMathSymbol{\ssfPsi}{0}{ssfletters}{'011}
\DeclareMathSymbol{\bsfOmega}{0}{bsfletters}{'012}
\DeclareMathSymbol{\ssfOmega}{0}{ssfletters}{'012}

\newcommand{\qednew}{\nobreak \ifvmode \relax \else
      \ifdim\lastskip<1.5em \hskip-\lastskip
      \hskip1.5em plus0em minus0.5em \fi \nobreak
      \vrule height0.75em width0.5em depth0.25em\fi}

\usepackage{bm,bbm}

\newcommand{\bone}{\mathbbm{1}}

\providecommand{\definitionname}{Definition}
\providecommand{\factname}{Fact}
\providecommand{\lemmaname}{Lemma}

\providecommand{\theoremname}{Theorem}

\makeatother

\providecommand{\definitionname}{Definition}
\providecommand{\factname}{Fact}
\providecommand{\lemmaname}{Lemma}
\providecommand{\theoremname}{Theorem}

\begin{document}
\title{The Entropy Method in Large Deviation Theory}
\author{Lei Yu\thanks{L. Yu is with the School of Statistics and Data Science, LPMC, KLMDASR,
and LEBPS, Nankai University, Tianjin 300071, China (e-mail: leiyu@nankai.edu.cn).
This work was supported by the NSFC grant 62101286 and the Fundamental
Research Funds for the Central Universities of China (Nankai University).}}
\maketitle
\begin{abstract}
This paper illustrates the power of the entropy method in addressing
problems from large deviation theory. We provide and review entropy
proofs for most fundamental results in large deviation theory, including
Cramer's theorem, the G\"artner--Ellis theorem, and Sanov's theorem.
Moreover, by the entropy method, we also strengthen Sanov's theorem
to the strong version. 
\end{abstract}

\section{Introduction }

Information theory was fundamentally established by the works of Harry
Nyquist and Ralph Hartley, in the 1920s, and Claude Shannon in the
1940s, which successfully addressed the theoretic limit of information
communication. Since then, information theory gradually influenced
other branches of mathematics. For example, the notion of (information)
entropy introduced by Shannon after he consulted von Neumann, the
father of the computer, was widely employed in probability theory,
functional analysis, ergodic theory, graph theory and so on. 

Information theory concerns realizing reliable and efficient communication
of a source over a noisy channel. Entropy and mutual information were
introduced by Shannon to determine the theoretical limit of such kind
of communication. Entropy and mutual information are special cases
of a more general concept, known as the \emph{relative entropy}. For
a nonnegative $\sigma$-finite measure $\mu$ and a probability measure
$Q$ defined on the same space such that $Q\ll\mu$, the relative
entropy of $Q$ with respect to (w.r.t.) $\mu$ is 
\[
D(Q\|\mu)=\int\log\left(\frac{\mathrm{d}Q}{\mathrm{d}\mu}\right)\mathrm{d}Q
\]
which is the expectation of the information density $\imath_{Q\|\mu}=\log\left(\frac{\mathrm{d}Q}{\mathrm{d}\mu}\right)$.
When $\mu$ is the counting measure on an countable set, $H(X)=-D(Q\|\mu)$
is the Shannon entropy of $X\sim Q$. When $\mu$ is the Lebesgue
measure on an Euclidean space, $h(X)=-D(Q\|\mu)$ is the differential
entropy of $X\sim Q$. The relative entropy was also widely used in
the case that $\mu$ is a probability measure. In this case, for two
probability measures $Q\ll P$, 
\begin{equation}
D(Q\|P)=\int\log\left(\frac{\mathrm{d}Q}{\mathrm{d}P}\right)\mathrm{d}Q.\label{eq:}
\end{equation}
If $Q$ is not absolutely continuous w.r.t. $P$, then $D(Q\|P):=+\infty$.
The mutual information between $X$ and $Y$ with $(X,Y)\sim P_{XY}$
is $I(X;Y)=D(P_{XY}\|P_{X}\otimes P_{Y})$. So, the relative entropy
is a general concept incorporating the entropy and mutual information.
The \emph{relative varentropy} of $Q$ with respect to (w.r.t.) $\mu$
is 
\[
\mathrm{V}(Q\|\mu)=\mathrm{Var}_{Q}\left[\imath_{Q\|\mu}(X)\right]=\mathrm{Var}_{Q}\left[\log\frac{\mathrm{d}Q}{\mathrm{d}\mu}(X)\right].
\]

\subsection{Organization }

In Section \ref{sec:Preliminaries:-Information-Proje}, we introduce
the concept of information projection (or shortly, I-projection),
which will be then used to characterize the asymptotic exponent in
the large deviation theory. In Section \ref{sec:Cramer's-Theorems},
we prove several versions of Cramer's theorem by the entropy method.
In particular, the entropy proof for the simple version of Cramer's
theorem in Section \ref{sec:Cramer's-Theorem-(Simple} is considered
as a paradigm to illustrate the power of the entropy method. Still
using the entropy method, we next generalize the Cramer's theorem
to the non-i.i.d. setting, and obtain the G\"artner--Ellis theorem
in Section \ref{sec:-Grtner=002013Ellis-Theorems}. We then focus
on another fundamental theorem in large deviation theory---Sanov's
theorem in Section \ref{sec:Sanov's-Theorems}. The entropy proof
for this theorem is given in Section \ref{sec:Sanov's-theorem}, and
the entropy proof for the strong version is given in Section \ref{sec:Sanov's-theorem-1}.
As an application, the entropy proof of Gibbs conditioning principle
is introduced in Section \ref{sec:Gibbs-conditioning-principle}. 

\subsection{Notations}

Let $\mathcal{X}$ be a Hausdorff topological space, and $\mathcal{B}_{\mathcal{X}}$
is the Borel $\sigma$-algebra on $\mathcal{X}$. Let  $P_{X}$ (or
shortly $P$) be a probability measure on $\mathcal{X}$. We also
use $Q_{X},R_{X}$ (or shortly $Q,R$) to denote another two probability
measures on $\mathcal{X}$. The probability measures $P_{X},Q_{X},R_{X}$
can be thought as the push-forward measures (or the distributions)
induced jointly by the same measurable function $X$ (random variable)
from an underlying measurable space to $\mathcal{X}$ and by different
probability measures $\P,\Q,\R$ defined on the underlying measurable
space. Without loss of generality, we assume that $X$ is the identity
map, and $\P,\Q,\R$ are the same as $P_{X},Q_{X},R_{X}$. So, $P_{X},Q_{X},R_{X}$
could be independently specified to arbitrary probability measures.
We say that all probability measures induced by the underlying measure
$\P$, together with the corresponding measurable spaces, constitute
the \emph{$\P$-system}. So, $P_{X}$ is in fact the distribution
of the random variable $X$ in the $\P$-system, where the letter
``$P$'' in the notation $P_{X}$ refers to the $\P$-system and
the subscript ``$X$'' refers to the random variable. When emphasizing
the random variables, we write $X\sim P_{X}$ to indicate that $X$
follows the distribution $P_{X}$ in the $\P$-system. 

We use $P_{X}^{\otimes n}$ to denote the $n$-fold product of $P_{X}$.
For a probability measure $P_{X}$ and a regular conditional distribution
(transition probability or Markov kernel) $P_{Y|X}$ from $\mathcal{X}$
to $\mathcal{Y}$, we denote $P_{X}P_{Y|X}$ as the joint probability
measure induced by $P_{X}$ and $P_{Y|X}$.  For a distribution $P_{X}$
on $\mathcal{X}$ and a measurable subset $A\subseteq\mathcal{X}$,
$P_{X}(\cdot|A)$ denotes the conditional probability measure given
$A$. For brevity, we write $P_{X}(x):=P_{X}(\{x\}),x\in\mathcal{X}$.
In particular, if $X\sim P_{X}$ is discrete, the restriction of $P_{X}$
to the set of singletons corresponds to the probability mass function
of $X$ in the $\P$-system. We use $Q_{X}\ll P_{X}$ to denote that
the distribution $Q_{X}$ is absolutely continuous w.r.t. $P_{X}$.
We use $X^{n}$ to denote a random vector $(X_{1},X_{2},...,X_{n})$
taking values on $(\mathcal{X}^{n},\mathcal{B}_{\mathcal{X}}^{\otimes n})$,
and use $x^{n}:=(x_{1},x_{2},...,x_{n})$ to denote its realization.
For an $n$-length vector $x^{n}$, we use $x^{i}$ to denote the
subvector consisting of the first $i$ components of $x^{n}$, and
$x_{i+1}^{n}$ to denote the subvector consisting of the last $n-i$
components. For a probability measure $P_{X^{n}}$ on $\mathcal{X}^{n}$,
we use $P_{X_{k}|X^{k-1}}$ to denote the regular conditional distribution
of $X_{k}$ given $X^{k-1}$ induced by $P_{X^{n}}$. For a measurable
function $f:\mathcal{X}\to\mathbb{R}$, sometimes we adopt the notation
$P_{X}(f):=\mathbb{E}_{P_{X}}[X]:=\int_{\mathcal{X}}f\ \mathrm{d}P_{X}$.
For a conditional probability measure $P_{X|Y}$, define the conditional
expectation operator induced by $P_{X|Y}$ as $P_{X|Y}(f)(y):=\int f\mathrm{d}P_{X|Y=y}$
for any measurable function $f:(\mathcal{X},\mathcal{B}_{\mathcal{X}})\to(\mathbb{R},\mathcal{B}_{\mathbb{R}})$
if the integral is well-defined for every $y$. When the random variable
$X$ is clear from the texture, we briefly denote $P_{X}$ as $P$.
For example, we denote the expectation $\mathbb{E}_{P_{X}}[X]$ as
$\mathbb{E}_{P}[X]$.

The relative entropy is defined in \eqref{eq:}. The conditional relative
entropy is defined as 
\[
D(Q_{X|W}\|P_{X|W}|Q_{W})=D(Q_{X|W}Q_{W}\|P_{X|W}Q_{W}).
\]

Given $n\geq1$, the \emph{empirical measure}  for a sequence $x^{n}\in\mathcal{X}^{n}$
is 
\[
\L_{x^{n}}:=\frac{1}{n}\sum_{i=1}^{n}\delta_{x_{i}}
\]
where $\delta_{x}$ is Dirac mass at the point $x\in\mathcal{X}$.

Denote $\Levy(P,Q)=\inf\{\delta>0:P(A)\le Q(A_{\delta})+\delta,\forall\textrm{ closed }A\subseteq\mathcal{Z}\}$
with $A_{\delta}:=\bigcup_{z\in A}\{z'\in\mathcal{Z}:d(z,z')<\delta\}$,
where $\mathcal{Z}$ is $\mathcal{X}$ or $\mathcal{Y}$. Here, $\Levy$
is known as the \emph{Lévy--Prokhorov metric} on $\mathcal{P}(\mathcal{Z})$
which is compatible with the weak topology on $\mathcal{P}(\mathcal{Z})$.
We use $B_{\delta}(R):=\{Q\in\mathcal{P}(\mathcal{Z}):\Levy(R,Q)<\delta\}$
and $B_{\delta]}(R):=\{Q\in\mathcal{P}(\mathcal{Z}):\Levy(R,Q)\leq\delta\}$
to respectively denote an open ball and a closed ball under the Lévy--Prokhorov
metric. We use $\overline{A}$, $A^{o}$, and $A^{c}:=\mathcal{Z}\backslash A$
to respectively denote the closure, interior, and complement of the
set $A\subseteq\mathcal{Z}$. Denote the sublevel set of the relative
entropy (or the divergence ``ball'') as $D_{\epsilon]}(P_{X}):=\{Q_{X}:D(Q_{X}\|P_{X})\le\epsilon\}$
for $\epsilon\ge0$. The Lévy--Prokhorov metric, the TV distance,
and the relative entropy admit the following relation: For any $Q_{X},P_{X}$,
\begin{equation}
\sqrt{2D(Q_{X}\|P_{X})}\ge\|Q_{X}-P_{X}\|_{\mathrm{TV}}\ge\Levy(Q_{X},P_{X}),\label{eq:D-TV-L}
\end{equation}
which implies for $\epsilon\ge0$, 
\begin{equation}
D_{\sqrt{2\epsilon}]}(P_{X})\subseteq B_{\epsilon]}(P_{X}).\label{eq:D-B}
\end{equation}
The first inequality in \eqref{eq:D-TV-L} is known as Pinsker's inequality,
and the second inequality follows by definition \cite{gibbs2002choosing}.

We denote $\inf\emptyset:=+\infty,\;\sup\emptyset:=-\infty$.  Given
two positive sequences $\{a_{n}\}_{n=1}^{\infty}$ and $\{b_{n}\}_{n=1}^{\infty}$,
we say that $a_{n}\lesssim b_{n}$ if $\limsup_{n\to\infty}a_{n}/b_{n}\le1$,
and $a_{n}\sim b_{n}$ if $a_{n}\lesssim b_{n}$ and $b_{n}\lesssim a_{n}$.

\section{\label{sec:Preliminaries:-Information-Proje}Preliminaries: Information
Projection }

For a set $\Gamma\subseteq\mathcal{P}(\mathcal{X})$, define 
\[
D(\Gamma\|P):=\inf_{R\in\Gamma}D(R\|P).
\]
Any probability measure $R^{*}$ attaining the infimum above is called
the \emph{information projection} or \emph{I-projection} of $P$ on
$\Gamma$ . 
\begin{lem}
If $\Gamma$ is closed in the weak topology, then the I-projection
exists. 
\end{lem}

\begin{proof}
This lemma follows by the lower semicontinuity of $R\mapsto D(R\|P)$
and the compactness of the sublevel sets of $R\mapsto D(R\|P)$ under
the weak topology. 
\end{proof}
If $\Gamma$ is convex and such $R^{*}$ exists, the convexity of
$\Gamma$ guarantees its uniqueness since $D(R\|P)$ is strictly convex
in $R$. Another important property of the I-projection is the following
equivalence. 
\begin{thm}
\cite[Theorem 2.2]{csiszar1975divergence} \label{thm:-A-probability}
A probability measure $Q\in\Gamma\cap D_{\infty}(P)$ is the I-projection
of $P$ on the convex set $\Gamma$ of probability measures iff every
$R\in\Gamma$ satisfies 
\begin{equation}
D(R\|P)\ge D(R\|Q)+D(Q\|P).\label{eq:-25}
\end{equation}
\end{thm}

The proof of this theorem is based on differentiating $D(P_{\lambda}\|P)$
w.r.t. $\lambda$ at $\lambda=0$, where $P_{\lambda}=\lambda P+(1-\lambda)Q$
with $\lambda\in[0,1]$; see details in \cite[Theorem 2.2]{csiszar1975divergence}.
The inequality \eqref{eq:-25} can be written as 
\begin{equation}
\int\left(\log\frac{\mathrm{d}Q}{\mathrm{d}P}\right)\mathrm{d}R\ge D(Q\|P),\label{eq:-25-2}
\end{equation}
or 
\begin{equation}
\int\left(\log\frac{\mathrm{d}Q}{\mathrm{d}P}\right)\mathrm{d}(R-Q)\ge0.\label{eq:-25-2-1}
\end{equation}
So, for the I-projection $Q$, the set 
\[
\mathcal{H}:=\left\{ R:\int\left(\log\frac{\mathrm{d}Q}{\mathrm{d}P}\right)\mathrm{d}(R-Q)=0\right\} 
\]
constitutes a ``supporting hyperplane'' of $\Gamma$. Moreover,
the function $1+\log\frac{\mathrm{d}Q}{\mathrm{d}P}$ can be seen
as the ``gradient'' of $Q\mapsto D(Q\|P)$, and $\mathcal{H}$ is
the ``tangent hyperplane'' of $Q\mapsto D(Q\|P)$. Analogously to
Euclidean spaces, considering $D(Q\|P)$ as a ``distance'', it holds
that $Q$ is the projection of $P$ to $\mathcal{H}$, and 
\begin{equation}
D(R\|P)=D(R\|Q)+D(Q\|P),\;\forall R\in\mathcal{H}.\label{eq:-25-3}
\end{equation}

As a consequence of Theorem \ref{thm:-A-probability}, when $\Gamma$
is specified by linear constraints, the I-projection can be written
out explicitly. To illustrate this point, we let $f:\mathcal{X}\to\mathbb{R}$
be a measurable function. Consider the following I-projection problem:
\begin{align*}
 & \gamma(\alpha):=\inf_{Q:Q(f)=\alpha}D(Q\|P).
\end{align*}
To give the I-projection for $\gamma(\alpha)$, we first introduce
the following lemma, which can be easily verified by definition.
\begin{lem}
\label{lem:identity} For a measurable function $f$ and $\lambda\in\mathbb{R}$
such that $P(e^{\lambda f})<+\infty$, define a probability measure
$Q_{\lambda}$ with density 
\[
\frac{\mathrm{d}Q_{\lambda}}{\mathrm{d}P}=\frac{e^{\lambda f}}{P(e^{\lambda f})},
\]
then for any $R$, 
\begin{align}
D(R\|P)-D(Q_{\lambda}\|P) & =D(R\|Q_{\lambda})+\lambda\left(R(f)-Q_{\lambda}(f)\right)\nonumber \\
 & \geq\lambda\left(R(f)-Q_{\lambda}(f)\right).\label{eq:-55-1-3}
\end{align}
\end{lem}

If there is $\lambda^{*}\in\mathbb{R}$ such that $P(e^{\lambda^{*}f})<+\infty$
and $Q_{\lambda^{*}}(f)=\alpha$, then by the lemma above, we have
for any $R\in\Gamma$, 
\begin{equation}
D(R\|P)=D(R\|Q_{\lambda^{*}})+D(Q_{\lambda^{*}}\|P),\label{eq:-35}
\end{equation}
which is an analogue of Pythagoras' theorem for information distance.
Further, by Theorem \ref{thm:-A-probability}, $Q_{\lambda^{*}}$
is the I-projection for $\gamma(\alpha)$. For this case, 
\[
\gamma(\alpha)=\lambda^{*}\alpha-\log P(e^{\lambda^{*}f}).
\]
In fact, the RHS above is equal to 
\begin{align}
 & \gamma^{*}(\alpha):=\sup_{\lambda\in\mathbb{R}}\lambda\alpha-\log P(e^{\lambda f}).\label{eq:-98}
\end{align}
In other words, when $\lambda^{*}$ exists, it attains the supremum
in \eqref{eq:-98}. We next show that $\gamma(\alpha)=\gamma^{*}(\alpha)$
always holds, even when $\lambda^{*}$ does not exist. 
\begin{thm}[Duality for the I-Projection (Equality Constraint)]
\label{thm:duality-I-projection} It holds that for any $\alpha\in\mathbb{R}$,
\begin{equation}
\gamma(\alpha)=\gamma^{*}(\alpha).\label{eq:-47-2-1}
\end{equation}
\end{thm}

The proof of this theorem is deferred to Section \ref{sec:Proof-of-Theorem},
since it is long and technical. 

We now consider the inequality constraint. Define 
\begin{align*}
 & \gamma_{+}(\alpha):=\inf_{Q:Q(f)\ge\alpha}D(Q\|P),
\end{align*}
and 
\begin{align*}
 & \gamma_{+}^{*}(\alpha):=\sup_{\lambda\ge0}\lambda\alpha-\log P(e^{\lambda f}).
\end{align*}
Following proof steps similar to those of Theorem \ref{thm:duality-I-projection},
we obtain the following duality. 
\begin{thm}[Duality for the I-Projection (Inequality Constraint)]
\label{thm:duality-I-projection-1} Let $f:\mathcal{X}\to\mathbb{R}$
be a measurable function. Then, it holds that for any $\alpha\in\mathbb{R}$,
\begin{equation}
\gamma_{+}(\alpha)=\gamma_{+}^{*}(\alpha).\label{eq:-47-2}
\end{equation}
\end{thm}

\section{\label{sec:Cramer's-Theorems}Cramer's Theorems}

\subsection{\label{sec:Cramer's-Theorem-(Simple}Cramer's Theorem (Simple Version)}

The entropy method is a mathematical method involving relative entropies
(or entropies). To illustrate the power of the entropy method, we
now provide an entropy proof of Cramer's theorem. Cramer's theorem
is the most fundamental result in the large deviation theory, which
characterizes the decay exponent of the tail probability for normalized
sum of i.i.d. real-valued random variables.
\begin{thm}[Cramer's Theorem]
\label{thm:Cramer} Let $X_{i}\sim P,i=1,2,\dots$ be i.i.d. real-valued
random variables. For any $\alpha\in\mathbb{R}$, 
\begin{align*}
\lim_{n\to\infty}-\frac{1}{n}\log\mathbb{P}\left\{ \frac{1}{n}\sum_{i=1}^{n}X_{i}\ge\alpha\right\}  & =\gamma(\alpha),
\end{align*}
where 
\begin{align*}
\gamma_{+}(\alpha) & :=\inf_{Q:\mathbb{E}_{Q}\left[X\right]\ge\alpha}D(Q\|P).
\end{align*}
\end{thm}

Here $\gamma$ is known as the \emph{rate function}, which admits
the following dual formula: For any $\alpha$, 
\begin{align*}
\gamma_{+}(\alpha) & =\gamma_{+}^{*}(\alpha):=\sup_{\lambda\ge0}\lambda\alpha-\log\mathbb{E}_{P}[e^{\lambda X}].
\end{align*}
Denote $\alpha_{\max}:=\sup\{\alpha:\gamma_{+}(\alpha)<+\infty\}$.
Since $\gamma_{+}$ is convex, nonnegative, and nondecreasing on $\mathbb{R}$,
it holds that $\gamma_{+}(\alpha)$ is continuous in $\alpha<\alpha_{\max}$.
It is easy to see that $\alpha_{\max}=\esssup_{P}X$ and $\gamma_{+}(\alpha_{\max})=-\log P\{\alpha_{\max}\}.$
\begin{proof}
Proof of ``$\ge$'': To this end, we can denote $A_{n}:=\left\{ x^{n}:\frac{1}{n}\sum_{i=1}^{n}x_{i}\ge\alpha\right\} $.
Then, the probability we are going to estimate is $P^{\otimes n}(A_{n})$.
Define an auxiliary probability measure $Q_{X^{n}}:=P^{\otimes n}(\cdot|A_{n})$.
It is easily verified that 
\begin{align}
-\frac{1}{n}\log P^{\otimes n}(A_{n}) & =\frac{1}{n}D(Q_{X^{n}}\|P^{\otimes n})\nonumber \\
 & \ge\frac{1}{n}\sum_{i=1}^{n}D(Q_{X_{i}}\|P)\nonumber \\
 & =D(Q_{X_{K}|K}\|P|Q_{K})\nonumber \\
 & \ge D(Q_{X}\|P),\label{eq:-1}
\end{align}
where the first inequality follows by the superadditivity of relative
entropy, the auxiliary random variable $K\sim Q_{K}:=\Unif[n]$ denotes
a random time index independent of $X^{n}$ (under both $P,Q$), and
in the last line, $X:=X_{K}$. On the other hand, since $Q_{X^{n}}$
is concentrated on $A_{n}$, we have $\mathbb{E}_{Q}\left[\frac{1}{n}\sum_{i=1}^{n}X_{i}\right]\ge\alpha$,
i.e., $\mathbb{E}_{Q}\left[X\right]\ge\alpha$. We hence have 
\begin{align*}
-\frac{1}{n}\log P^{\otimes n}(A_{n}) & \geq\gamma_{+}(\alpha).
\end{align*}

Proof of ``$\le$'': It is verified that 
\begin{align}
-\frac{1}{n}\log P^{\otimes n}(A_{n}) & =\inf_{Q_{X^{n}}:Q_{X^{n}}(A_{n})=1}\frac{1}{n}D(Q_{X^{n}}\|P^{\otimes n}),\label{eq:-39}
\end{align}
where the optimal $Q_{X^{n}}$ attaining the infimum is $P^{\otimes n}(\cdot|A_{n})$
since for any $Q_{X^{n}}$ concentrated on $A_{n}$, 
\[
D(Q_{X^{n}}\|P^{\otimes n})=D(Q_{X^{n}}\|P^{\otimes n}(\cdot|A_{n}))-\frac{1}{n}\log P^{\otimes n}(A_{n}).
\]

Given any $Q_{X}$ such that $\mathbb{E}_{Q}\left[X\right]>\alpha$,
denote auxiliary probability measures $R_{X^{n}}:=Q_{X}^{\otimes n}(\cdot|A_{n})$
and $\bar{R}_{X^{n}}:=Q_{X}^{\otimes n}(\cdot|A_{n}^{c})$. Denote
$p_{n}:=Q_{X}^{\otimes n}(A_{n})$, which converges to $1$ by the
law of large numbers (LLN). Furthermore, 
\begin{align*}
nD(Q_{X}\|P) & =p_{n}D(R_{X^{n}}\|P^{\otimes n})+(1-p_{n})D(\bar{R}_{X^{n}}\|P^{\otimes n})-H_{2}(p_{n})\\
 & \ge p_{n}D(R_{X^{n}}\|P^{\otimes n})-H_{2}(p_{n}),
\end{align*}
where $H_{2}:t\in[0,1]\mapsto-t\log_{2}t-(1-t)\log_{2}(1-t)$ is the
binary entropy function and the equality above can be easily verified
by definition. That is, 
\[
D(R_{X^{n}}\|P^{\otimes n})\leq\frac{nD(Q_{X}\|P)+H_{2}(p_{n})}{p_{n}}.
\]
Since $R_{X^{n}}$ is a feasible solution to the infimization in \eqref{eq:-39},
\begin{align*}
\limsup_{n\to\infty}-\frac{1}{n}\log P^{\otimes n}(A_{n}) & \le\limsup_{n\to\infty}\frac{nD(Q_{X}\|P)+H_{2}(p_{n})}{np_{n}}\\
 & =D(Q_{X}\|P).
\end{align*}
Since $Q_{X}$ satisfying $\mathbb{E}_{Q}\left[X\right]>\alpha$ is
arbitrary, 
\begin{align*}
\limsup_{n\to\infty}-\frac{1}{n}\log P^{\otimes n}(A_{n}) & \le\inf_{Q_{X}:\mathbb{E}_{Q}\left[X\right]>\alpha}D(Q_{X}\|P)\\
 & \le\lim_{\delta\downarrow0}\gamma_{+}(\alpha+\delta).
\end{align*}
By the continuity, for $\alpha\neq\alpha_{\max}$, $\limsup_{n\to\infty}-\frac{1}{n}\log P^{\otimes n}(A_{n})\le\gamma_{+}(\alpha)$,
and hence $\lim_{n\to\infty}-\frac{1}{n}\log P^{\otimes n}(A_{n})=\gamma_{+}(\alpha)$.

For $\alpha=\alpha_{\max}$, $\gamma_{+}(\alpha_{\max})=-\log P_{X}\{\alpha_{\max}\}$.
For this case, $P^{\otimes n}(A_{n})\ge P\{\alpha_{\max}\}^{n}$,
and hence $\lim_{n\to\infty}-\frac{1}{n}\log P^{\otimes n}(A_{n})=\gamma_{+}(\alpha)$
still holds. 
\end{proof}
In fact, common proofs for Cramer's theorem, e.g., the one given in
\cite{Dembo}, are from the dual perspective, for which the resultant
rate function is expressed by the Fenchel--Legendre transform of
the logarithmic moment generating function. However, the entropy proof
here is from the primal perspective, for which the resultant rate
function is characterized by an information projection problem (the
minimization of the relative entropy over a convex set). Furthermore,
the entropy proof here is short, since we do not need to divide the
proof into many cases, e.g., whether the expectation of $X_{i}$ exists.
The entropy proof neither requires the technique of change of measure,
since the auxiliary probability measure $Q_{X}$ in the proof of ``$\le$''
is chosen independently of the original probability measure $P^{\otimes n}$,
which plays the role of changing of measure. 

Generally speaking, the entropy method  typically consists of three
steps: 
\begin{enumerate}
\item First, introduce auxiliary probability measures (or auxiliary random
variables), 
\item Then, express the problem in terms of relative entropies of these
auxiliary probability measures, 
\item Lastly, derive bounds by using properties of relative entropies. 
\end{enumerate}
Moreover, for the first step, in a probability measure space, if the
extreme problem that we consider is about sets, then the auxiliary
measures are usually defined as conditional probability measures given
these sets; if the extreme problem is about nonnegative integrable
functions, then the auxiliary measures are usually defined as probability
measures with densities proportional to these functions (or their
variants). The first step is unnecessary if the probability measures
that we want are already given in the problem. 

The example above illustrates the power of the entropy method, and
this example is not an exceptional case. In fact, this method is simple,
general, and powerful in the sense that it does not only apply to
many probabilistic, combinatorial, and functional-analytic problems,
but also works for general probability measure spaces and usually
yields exponentially tight bounds.

\subsection{\label{sec:Cramer's-theorem}Cramer's Theorem (General Version)}

The simple version of Cramer's theorem was proven in the previous
section. We now prove the general version of this theorem.

Let $\mathcal{X}=\mathbb{R}$. Consider a random vector $X^{n}$,
consisting of i.i.d. real-valued random variables $X_{i}\sim P,i\in[n]$.
Let $\mu_{n}$ denote the law of $\hat{S}_{n}=\frac{1}{n}\sum_{i=1}^{n}X_{i}$.
\begin{defn}
A sequence of probability measures $\{\nu_{n}\}$ on $(\mathcal{X},\mathcal{B})$
is said to satisfy the \emph{large deviation principle} with a rate
function $I$ if, for all $\Gamma\in\mathcal{B}$, 
\begin{equation}
\inf_{x\in\overline{\Gamma}}I(x)\le\liminf_{n\to\infty}-\frac{1}{n}\log\nu_{n}(\Gamma)\le\limsup_{n\to\infty}-\frac{1}{n}\log\nu_{n}(\Gamma)\le\inf_{x\in\Gamma^{o}}I(x).\label{eq:-93}
\end{equation}
\end{defn}

Define 
\begin{align}
\gamma(\alpha) & :=\inf_{Q:\mathbb{E}_{Q}\left[X\right]=\alpha}D(Q\|P)\label{eq:gamma}\\
 & =\sup_{\lambda\in\mathbb{R}}\lambda\alpha-\log\mathbb{E}_{P}[e^{\lambda X}],\nonumber 
\end{align}
where the duality follows by Theorem \ref{thm:duality-I-projection}.
\begin{thm}[Cramer's Theorem]
\cite[Theorem 2.2.3]{Dembo} \label{thm:Cramer-1} The sequence of
measures $\{\mu_{n}\}$ satisfies the LDP with the convex rate function
$\gamma$, namely: \\
 (a) For any closed set $F\subseteq\mathbb{R}$, 
\[
\liminf_{n\to\infty}-\frac{1}{n}\log\mu_{n}(F)\ge\inf_{x\in F}\gamma(x).
\]
(b) For any open set $G\subseteq\mathbb{R}$, 
\[
\limsup_{n\to\infty}-\frac{1}{n}\log\mu_{n}(G)\le\inf_{x\in G}\gamma(x).
\]
\end{thm}

\begin{proof}
Proof of (a): We extend the domain of $\gamma$ to $\bar{\mathbb{R}}=\mathbb{R}\cup\{\pm\infty\}$
by the continuous extension. Let $\alpha^{*}\in\bar{\mathbb{R}}$
be the minimum point of $\gamma$, i.e., $\gamma(\alpha^{*})\le\gamma(\alpha),\forall\alpha\in\bar{\mathbb{R}}$.
For a set $F\subseteq\mathbb{R}$, let $F_{1}:=F\cap(-\infty,\alpha^{*})$
and $F_{2}:=F\cap(\alpha^{*},+\infty)$. Let $\alpha_{1},\alpha_{2}\in\bar{\mathbb{R}}$
be such that $\alpha_{1}=\sup F_{1}$ and $\alpha_{2}=\inf F_{2}$.
Define 
\begin{align*}
\gamma_{-}(\alpha) & :=\inf_{Q:\mathbb{E}_{Q}\left[X\right]\le\alpha}D(Q\|P)\\
\gamma_{+}(\alpha) & :=\inf_{Q:\mathbb{E}_{Q}\left[X\right]\ge\alpha}D(Q\|P).
\end{align*}
Observe that $\gamma$ is convex, and $\gamma_{-}(\alpha)=\inf_{x\le\alpha}\gamma(x)$
and $\gamma_{+}(\alpha)=\inf_{x\ge\alpha}\gamma(x)$. So, 
\begin{align*}
\gamma_{-}(\alpha) & =\begin{cases}
\gamma(\alpha) & \alpha\le\alpha^{*}\\
\gamma(\alpha^{*}) & \alpha>\alpha^{*}
\end{cases},\\
\gamma_{+}(\alpha) & =\begin{cases}
\gamma(\alpha) & \alpha\ge\alpha^{*}\\
\gamma(\alpha^{*}) & \alpha<\alpha^{*}
\end{cases}.
\end{align*}

Then, 
\begin{align*}
-\frac{1}{n}\log\mu_{n}(F) & \ge-\frac{1}{n}\log\left(\mu_{n}((-\infty,\alpha_{1}])+\mu_{n}([\alpha_{2},\infty))\right)\\
 & \ge\min\{\gamma_{-}(\alpha_{1}),\gamma_{+}(\alpha_{2})\}-\epsilon_{n}\\
 & =\min\{\gamma(\alpha_{1}),\gamma(\alpha_{2})\}-\epsilon_{n}\\
 & \ge\inf_{x\in F}\gamma(x)-\epsilon_{n},
\end{align*}
where $\epsilon_{n}=\frac{1}{n}\log2$, and the second inequality
follows by Theorem \ref{thm:Cramer} (more specifically, the proof
of Theorem \ref{thm:Cramer}).

Proof of (b): For any open set $G\subseteq\mathbb{R}$, let $A_{n}=\{x^{n}:\frac{1}{n}\sum_{i=1}^{n}x_{i}\in G\}$.
Given $Q$ such that $\mathbb{E}_{Q}\left[X\right]\in G$, denote
$R_{X^{n}}:=Q^{\otimes n}(\cdot|A_{n})$ and $\bar{R}_{X^{n}}:=Q^{\otimes n}(\cdot|A_{n}^{c})$.
Denote $p_{n}:=Q^{\otimes n}(A_{n})$, which converges to $1$ by
the law of large numbers (LLN). Furthermore, following steps similar
to the proof of Theorem \ref{thm:Cramer}, it holds that 
\begin{align*}
-\frac{1}{n}\log P^{\otimes n}(A_{n}) & \le\inf_{Q:\mathbb{E}_{Q}\left[X\right]\in G}\frac{nD(Q\|P)+H_{2}(p_{n})}{np_{n}}\\
 & \to\inf_{Q:\mathbb{E}_{Q}\left[X\right]\in G}D(Q\|P)\\
 & =\inf_{x\in G}\gamma(x).
\end{align*}
\end{proof}

\subsection{\label{sec:Cramer's-Theorem-Rd}Cramer's Theorem for $\mathbb{R}^{d}$ }

We now consider $\mathcal{X}=\mathbb{R}^{d}$. Consider a random sequence
$X^{n}$, consisting of i.i.d. random vectors $X_{i}\sim P,i\in[n]$.
Let $\mu_{n}$ denote the law of $\hat{S}_{n}=\frac{1}{n}\sum_{i=1}^{n}X_{i}$.
We next extend the general version of Cramer's theorem to this case.

Denote $\Lambda(\lambda):=\log\mathbb{E}_{P}[e^{\left\langle \lambda,X\right\rangle }]$.
Denote its effective domain $\mathcal{D}_{\Lambda}:=\left\{ \lambda\in\mathbb{R}^{d}:\Lambda(\lambda)<+\infty\right\} $.
For simplicity, suppose $\mathcal{D}_{\Lambda}=\mathbb{R}^{d}$. So,
$\Lambda(be_{i})\to+\infty$ as $b\to\infty$, where $e_{i}$ is a
$d$-length vector with $1$ as the $i$th coordinate and $0$'s as
other coordinates. 

Define for $\alpha\in\mathbb{R}^{d}$, 
\begin{align}
\gamma(\alpha) & :=\inf_{Q:\mathbb{E}_{Q}\left[X\right]=\alpha}D(Q\|P)\nonumber \\
 & =\sup_{\lambda\in\mathbb{R}^{d}}\left\langle \lambda,\alpha\right\rangle -\Lambda(\lambda).\label{eq:-3}
\end{align}

\begin{thm}[Cramer's Theorem for $\mathbb{R}^{d}$]
\cite[Theorem 2.2.30]{Dembo} \label{thm:Cramer-1-1} The sequence
of measures $\{\mu_{n}\}$ satisfies the LDP with the convex rate
function $\gamma$, namely: \\
 (a) For any closed set $F\subseteq\mathbb{R}^{d}$, 
\[
\liminf_{n\to\infty}-\frac{1}{n}\log\mu_{n}(F)\ge\inf_{x\in F}\gamma(x).
\]
(b) For any open set $G\subseteq\mathbb{R}^{d}$, 
\[
\limsup_{n\to\infty}-\frac{1}{n}\log\mu_{n}(G)\le\inf_{x\in G}\gamma(x).
\]
\end{thm}

\begin{defn}
A family of probability measures $\mu_{n}$ on $\mathcal{X}$ is
\emph{exponentially tight} if for every  $b\in(0,\infty)$, there
is a compact set $\mathcal{K}_{b}\subseteq\mathcal{P}(\mathcal{X})$
such that 
\begin{equation}
\liminf_{n\to\infty}-\frac{1}{n}\log\mu_{n}(\mathcal{K}_{b}^{c})\ge b.\label{eq:-94-2}
\end{equation}
\end{defn}

\begin{proof}[Proof of Theorem \ref{thm:Cramer-1-1}]
 Proof of (a): Suppose $F$ is a closed ball which is obviously convex.
Then, we denote $A_{n}:=\left\{ x^{n}:\frac{1}{n}\sum_{i=1}^{n}x_{i}\in F\right\} $.
Define $Q_{X^{n}}:=P^{\otimes n}(\cdot|A_{n})$. Similarly to \eqref{eq:-1},
it holds that 
\begin{align*}
-\frac{1}{n}\log P^{\otimes n}(A_{n}) & \ge D(Q_{X}\|P),
\end{align*}
where $X:=X_{K}$ and the auxiliary random variable $K\sim Q_{K}:=\Unif[n]$
denotes a random time index independent of $X^{n}$ (under both $P,Q$).
On the other hand, since $Q_{X^{n}}$ is concentrated on $A_{n}$,
by Jensen's inequality, we have $\mathbb{E}_{Q}\left[\frac{1}{n}\sum_{i=1}^{n}X_{i}\right]\in F$,
i.e., $\mathbb{E}_{Q}\left[X\right]\in F$. We hence have 
\begin{align*}
-\frac{1}{n}\log P^{\otimes n}(A_{n}) & \geq\inf_{x\in F}\gamma(x).
\end{align*}

Since compact sets can be covered by an appropriate finite collection
of small enough balls, Statement (a) for compact sets follows by the
union of events bound and the lower semicontinuity of the rate function
(seen from \eqref{eq:-3}). We can also extend Statement (a) to closed
sets by the fact that $\mu_{n}$ is an exponentially tight family
of probability measures. 
\begin{lem}[Exponential Tightness]
\label{lem:For-each-,-1} \cite[pp. 38-39]{Dembo} 
\begin{equation}
\lim_{\rho\to\infty}\liminf_{n\to\infty}-\frac{1}{n}\log\mu_{n}(\mathbb{R}^{d}\backslash[-\rho,\rho]^{d})=+\infty.\label{eq:-94-1-1}
\end{equation}
In other words, $\mu_{n}$ is exponentially tight.
\end{lem}

Proof of (b): Before proving (b), we first prove the following lemma.
 Here without loss of generality, we equip $\mathbb{R}^{d}$ with
the metric $\|x-y\|_{\infty}$.
\begin{lem}[Concentration on Balls]
\label{lem:-exponentially-fast-1} For any $Q$, 
\begin{align}
\liminf_{n\to\infty}-\frac{1}{n}\log Q^{\otimes n}\{x^{n}:\frac{1}{n}\sum_{i=1}^{n}x_{i}\in B_{\epsilon}(\alpha)^{c}\} & >0,\label{eq:-95-1-1-1}
\end{align}
where $\alpha:=\mathbb{E}_{Q}\left[X\right]$. In other words, $\mathbb{P}\{\frac{1}{n}\sum_{i=1}^{n}X_{i}\in B_{\epsilon}(\alpha)^{c}\}\to0$
exponentially fast as $n\to\infty$. 
\end{lem}

This lemma is proven as follows. We have 
\begin{align*}
Q^{\otimes n}\{x^{n}:\frac{1}{n}\sum_{i=1}^{n}x_{i}\in B_{\epsilon}(\alpha)^{c}\} & \le Q^{\otimes n}\{x^{n}:\|\frac{1}{n}\sum_{i=1}^{n}x_{i}-\alpha\|_{\infty}\ge\epsilon\}\\
 & \le\sum_{j=1}^{d}Q_{X(j)}^{\otimes n}\{x^{n}(j):|\frac{1}{n}\sum_{i=1}^{n}x_{i}(j)-\alpha(j)|\ge\epsilon\}\\
 & \to0
\end{align*}
exponentially fast as $n\to\infty$, where $x(j)$ denotes the $j$-th
coordinate of $x$ and the convergence in the last line follows by
the simple version of Cramer's theorem in Theorem \ref{thm:Cramer}.
So, Lemma \ref{lem:-exponentially-fast-1} holds.

We turn back to proving (b). It suffices to assume that $G$ is open.
Denote $A_{n}:=\left\{ x^{n}:\frac{1}{n}\sum_{i=1}^{n}x_{i}\in G\right\} $
and choose $Q$ such that $\alpha:=\mathbb{E}_{Q}\left[X\right]\in G$.
Then, $B_{\epsilon}(\alpha)\subseteq G$ for sufficiently small $\epsilon>0$.
Denote $p_{n}:=Q^{\otimes n}(A_{n})$. Then, by Lemma \ref{lem:-exponentially-fast-1},
\begin{align*}
p_{n} & \ge Q^{\otimes n}\{x^{n}:\frac{1}{n}\sum_{i=1}^{n}x_{i}\in B_{\epsilon}(\alpha)\}\\
 & =1-Q^{\otimes n}\{x^{n}:\frac{1}{n}\sum_{i=1}^{n}x_{i}\in B_{\epsilon}(\alpha)^{c}\}\\
 & \to1\textrm{ as }n\to\infty.
\end{align*}

Following steps similar to those in the proof of Theorem \ref{thm:Cramer},
we have 
\begin{align*}
\limsup_{n\to\infty}-\frac{1}{n}\log P^{\otimes n}(A_{n}) & \le\inf_{Q_{X}:\mathbb{E}_{Q}\left[X\right]\in G}D(Q_{X}\|P)=\inf_{x\in G}\gamma(x).
\end{align*}
\end{proof}

\subsection{\label{sec:Strong-LD-and}Strong Cramer's Theorem }

The simple version of Cramer's theorem given in Section \ref{sec:Cramer's-Theorem-(Simple}
can be further strengthened. The resultant result is known as the
\emph{strong Cramer's theorem} \cite{Bah60}. We next present the
proof in \cite{Bah60} for the strong Cramer's theorem. Although the
proof is not new, the strong Cramer's theorem is rewritten in terms
of relative entropies and relative varentropies. 

Let $\mathcal{X}=\mathbb{R}$. Consider a random vector $X^{n}$,
consisting of i.i.d. real-valued random variables $X_{i}\sim P,i\in[n]$.
Suppose that $Q$ attains $\gamma(\alpha)$ defined in \eqref{eq:gamma}
which satisfies 
\[
\frac{\mathrm{d}Q}{\mathrm{d}P}(x)=\frac{e^{\lambda^{*}x}}{\mathbb{E}_{P}[e^{\lambda^{*}X}]}
\]
for some $\lambda^{*}\in\mathbb{R}$ such that $\mathbb{E}_{Q}\left[X\right]=\alpha$.
Then, we have 
\begin{align}
 & \mathbb{P}\{\sum_{i=1}^{n}X_{i}\ge n\alpha\}\nonumber \\
 & =\int\bone_{\{\sum_{i=1}^{n}x_{i}\ge n\alpha\}}\frac{\mathrm{d}P^{\otimes n}}{\mathrm{d}Q^{\otimes n}}\mathrm{d}Q^{\otimes n}\\
 & =e^{-nD(Q\|P)}\int\bone_{\{\sum_{i=1}^{n}x_{i}\ge n\alpha\}}e^{-\left(\log\frac{\mathrm{d}Q^{\otimes n}}{\mathrm{d}P^{\otimes n}}-nD(Q\|P)\right)}\mathrm{d}Q^{\otimes n}\nonumber \\
 & =e^{-nD(Q\|P)}\mathbb{E}\left[\bone_{\{\sum_{i=1}^{n}W_{i}\ge0\}}e^{-\sum_{i=1}^{n}W_{i}}\right],\label{eq:-15}
\end{align}
where $W_{i}:=\log\frac{\mathrm{d}Q}{\mathrm{d}P}(X_{i})-D(Q\|P)$
with i.i.d. $X_{i}\sim Q$. Note that $\mathbb{E}\left[W_{i}\right]=0$
and $\mathrm{Var}\left[W_{i}\right]=\mathrm{V}(Q\|P)=\lambda^{*2}\mathrm{Var}_{Q}\left[X\right]$.

To estimate the expectation in \eqref{eq:-15}, we need to use the
central limit theorems. Intuitively, $\frac{1}{\sqrt{n\mathrm{V}(Q\|P)}}\sum_{i=1}^{n}W_{i}$
asymptotically follows the standard normal distribution. By simply
replacing it with a standard normal random variable $Z$, we obtain
\[
\mathbb{E}\left[\bone_{\{Z\ge0\}}e^{-\sqrt{n\mathrm{V}(Q\|P)}\cdot Z}\right]\sim\frac{1}{\sqrt{2\pi n\mathrm{V}(Q\|P)}}.
\]
In fact, this intuition is true when $X_{i}\sim P$ are non-lattice.
When $X_{i}\sim P$ are lattice, an additional factor will appear
at the RHS. These claims, stated in the following lemma, can be proven
by using Berry--Ess\'een expansion; see e.g., \cite{Dembo}. 
\begin{lem}
\label{lem:expect} For i.i.d. $W_{i}$ with mean zero and variance
$V$, it holds that 
\[
\mathbb{E}\left[\bone_{\{\sum_{i=1}^{n}W_{i}\ge0\}}e^{-\sum_{i=1}^{n}W_{i}}\right]\sim\frac{c}{\sqrt{2\pi nV}},
\]
where $c=1$ if $W_{i}$ are non-lattice, and $c=\frac{d}{1-e^{-d}}$
if $W_{i}$ are lattice with maximal step\footnote{That is, for some $x_{0},d$, the random variable $d^{-1}(W_{1}-w_{0})$
is (a.s.) an integer number, and $d$ is the largest number with this
property.} $d$ and\footnote{The condition $0<\mathbb{P}(W_{1}=0)<1$ implies that $w_{0}/d$ is
an integer and that $V>0.$} $0<\mathbb{P}(W_{1}=0)<1$. 
\end{lem}

So, we obtain the following strong LD theorem. Note that if $X_{i}\sim P$
are lattice with maximal step $d$, then $W_{i}$ are lattice with
maximal step $\lambda^{*}d$. 
\begin{thm}[Strong Cramer's Theorem]
\cite{Bah60} \label{thm:LD-2-1} Let $\alpha>0$. Suppose that $Q$
attain $\gamma(\alpha)$. Then, 
\[
\mathbb{P}\left\{ \sum_{i=1}^{n}X_{i}\ge n\alpha\right\} \sim\frac{c}{\sqrt{2\pi n\mathrm{V}(Q\|P)}}e^{-nD(Q\|P)},
\]
where $c=1$ if $X_{i}\sim P$ are non-lattice, and $c=\frac{\lambda^{*}d}{1-e^{-\lambda^{*}d}}$
if $X_{i}\sim P$ are lattice with maximal step $d$ and $0<\mathbb{P}(X_{1}=\alpha)<1$. 
\end{thm}

\section{\label{sec:-Grtner=002013Ellis-Theorems} General Principle and
G\"artner--Ellis Theorem }

\subsection{General Principle}

We now extend Cramer's theorem to the non-i.i.d. setting. Suppose
that $\mathcal{X}$ is a metric space, and $\mathcal{B}_{\mathcal{X}}$
is the Borel $\sigma$-algebra on $\mathcal{X}$. Consider a sequence
of probability measures $\{\mu_{n}\}$ on $\mathcal{X}$. Define for
$x\in\mathcal{X}$, 
\begin{align*}
\gamma_{-}(x) & :=\lim_{\epsilon\downarrow0}\liminf_{n\to\infty}E_{n,\epsilon}(x)\\
\gamma_{+}(x) & :=\lim_{\epsilon\downarrow0}\limsup_{n\to\infty}E_{n,\epsilon}(x),
\end{align*}
where 
\[
E_{n,\epsilon}(x):=-\frac{1}{n}\log\mu_{n}(B_{\epsilon}(x)).
\]
Note that $\gamma_{-}$ and $\gamma_{+}$ can be written in terms
of relative entropies since 
\begin{equation}
E_{n,\epsilon}(x)=\inf_{\nu_{n}:\nu_{n}(B_{\epsilon}(x))=1}\frac{1}{n}D(\nu_{n}\|\mu_{n}).\label{eq:-6}
\end{equation}

\begin{thm}[General Principle]
 \label{thm:general} Assume that $\{\mu_{n}\}$ is an exponentially
tight sequence of probability measures. Then, it holds that:\\
 (a) For any closed set $F\subseteq\mathcal{X}$, 
\[
\liminf_{n\to\infty}-\frac{1}{n}\log\mu_{n}(F)\ge\inf_{x\in F}\gamma_{-}(x).
\]
(b) For any open set $G\subseteq\mathcal{X}$, 
\[
\limsup_{n\to\infty}-\frac{1}{n}\log\mu_{n}(G)\le\inf_{x\in G}\gamma_{+}(x).
\]
(c) Suppose $\gamma_{-}(x)=\gamma_{+}(x)=:\gamma(x)$ for all $x\in\mathcal{X}$.
Then, $\{\mu_{n}\}$ satisfies the LDP with the rate function $\gamma$. 
\end{thm}

In fact, the parameter $n$ can be replaced by any positive $a_{n}$
such that  $a_{n}\to\infty$ as $n\to\infty$, or even replaced by
$1/\epsilon$  in which case, all limits are taken as $\epsilon\downarrow0$.
\begin{proof}
Proof of (a): Since any compact set $F$ can be covered by an appropriate
finite collection of small enough balls $\{B_{\epsilon}(x_{i})\}_{i=1}^{k}$
with $x_{i}\in F$, Statement (a) for compact sets follows by the
union of events bound. Namely, 
\begin{align*}
\mu_{n}(F) & \le\sum_{i=1}^{k}\mu_{n}(B_{\epsilon}(x_{i}))\le k\max_{1\le i\le k}\mu_{n}(B_{\epsilon}(x_{i})),
\end{align*}
and hence, 
\begin{align}
\liminf_{n\to\infty}-\frac{1}{n}\log\mu_{n}(F) & \geq\min_{1\le i\le k}\liminf_{n\to\infty}-\frac{1}{n}\log\mu_{n}(B_{\epsilon}(x_{i}))\nonumber \\
 & \geq\inf_{x\in F}\liminf_{n\to\infty}-\frac{1}{n}\log\mu_{n}(B_{\epsilon}(x)).\label{eq:-5}
\end{align}
Let $y_{i}\in F$  attain the infimum in \eqref{eq:-5} within a
gap $1/i$. Passing to a subsequence, we assume $y_{i}\to y^{*}\in F$
as $i\to\infty$. So, the bound in \eqref{eq:-5} is equal to 
\begin{align}
\lim_{i\to\infty}\liminf_{n\to\infty}-\frac{1}{n}\log\mu_{n}(B_{\epsilon}(y_{i})) & .\label{eq:-5-1}
\end{align}
 For sufficiently large $i$, $d(y^{*},y_{i})<\epsilon$. So, \eqref{eq:-5-1}
is further lower bounded by 
\begin{align}
\liminf_{n\to\infty}-\frac{1}{n}\log\mu_{n}(B_{2\epsilon}(y^{*})) & .\label{eq:-5-1-1}
\end{align}
Substituting these into \eqref{eq:-5} yields 
\begin{align*}
\liminf_{n\to\infty}-\frac{1}{n}\log\mu_{n}(F) & \geq\liminf_{n\to\infty}-\frac{1}{n}\log\mu_{n}(B_{2\epsilon}(y^{*}))\\
 & \ge\gamma_{-}(y^{*})\\
 & \ge\inf_{x\in F}\gamma_{-}(x).
\end{align*}

Finally, by the assumption that $\mu_{n}$ is an exponentially tight
family of probability measures, we can extend Statement (a) from compact
sets to all closed sets. 

Proof of (b): For open set $G$, let $x\in G$. Then, $B_{\epsilon}(x)\subseteq G$
for sufficiently small $\epsilon>0$. We have 
\begin{align*}
\limsup_{n\to\infty}-\frac{1}{n}\log\mu_{n}(G) & \le\lim_{\epsilon\downarrow0}\limsup_{n\to\infty}-\frac{1}{n}\log\mu_{n}(B_{\epsilon}(x))\\
 & =\gamma_{+}(x).
\end{align*}
Since $x\in G$ is arbitrary, 
\[
\limsup_{n\to\infty}-\frac{1}{n}\log\mu_{n}(G)\le\inf_{x\in G}\gamma_{+}(x).
\]
 
\end{proof}

\subsection{G\"artner--Ellis Theorem}

We now suppose that $\mathcal{X}=\mathbb{R}^{d}$.  Let $\mu_{n}$
denote the law of $Z_{n}\in\mathbb{R}^{d}$. Denote $\Lambda_{n}(\lambda):=\log\mathbb{E}[e^{\left\langle \lambda,Z_{n}\right\rangle }]$.
Denote $\Lambda(\lambda):=\lim_{n\to\infty}\frac{1}{n}\Lambda_{n}(n\lambda)$
where the limit is supposed to exist. Denote its effective domain
$\mathcal{D}_{\Lambda}:=\left\{ \lambda\in\mathbb{R}^{d}:\Lambda(\lambda)<+\infty\right\} $.
Let $\Lambda^{*}(\cdot)$ be the Fenchel--Legendre transform of $\Lambda(\cdot)$,
with $\mathcal{D}_{\Lambda^{*}}:=\left\{ x\in\mathbb{R}^{d}:\Lambda^{*}(x)<+\infty\right\} $.
Define for $x\in\mathbb{R}^{d}$,
\begin{align*}
\gamma_{n}(x) & :=\inf_{\nu_{n}:\mathbb{E}_{\nu_{n}}[X]=x}\frac{1}{n}D(\nu_{n}\|\mu_{n})\\
 & =\sup_{\lambda\in\mathbb{R}^{d}}\left\langle \lambda,x\right\rangle -\frac{1}{n}\Lambda_{n}(n\lambda).
\end{align*}

\textbf{Assumption}: For each $\lambda\in\mathbb{R}^{d}$, the logarithmic
moment generating function $\Lambda(\lambda)$ exists as an extended
real number. Further, the origin belongs to the interior of $\mathcal{D}_{\Lambda}$. 
\begin{defn}
$y\in\mathbb{R}^{d}$ is an exposed point of $\Lambda^{*}$ if for
some $\lambda\in\mathbb{R}^{d}$ and all $x\neq y$, 
\begin{equation}
\left\langle \lambda,y\right\rangle -\Lambda^{*}(y)>\left\langle \lambda,x\right\rangle -\Lambda^{*}(x).\label{eq:-2}
\end{equation}
 $\lambda$ in \eqref{eq:-2} is called an exposing hyperplane.
\end{defn}

\begin{thm}[G\"artner--Ellis Theorem]
\cite[Theorem 2.3.6]{Dembo} \label{thm:GE} Let the assumption above
hold. \\
 (a) For any closed set $F\subseteq\mathbb{R}^{d}$, 
\[
\liminf_{n\to\infty}-\frac{1}{n}\log\mu_{n}(F)\ge\inf_{x\in F}\Lambda^{*}(x).
\]
(b) For any open set $G\subseteq\mathbb{R}^{d}$, 
\[
\limsup_{n\to\infty}-\frac{1}{n}\log\mu_{n}(G)\le\inf_{x\in G\cap\mathcal{F}}\Lambda^{*}(x),
\]
where $\mathcal{F}$ is the set of exposed points of $\Lambda^{*}$
whose exposing hyperplane belongs to $\mathcal{D}_{\Lambda}^{o}$.
\\
\end{thm}

\begin{defn}
A convex function $\Lambda:\mathbb{R}^{d}\to(-\infty,\infty]$ is
essentially smooth if: (a) $\mathcal{D}_{\Lambda}^{o}$ is non-empty.
(b) $\Lambda(\cdot)$ is differentiable throughout $\mathcal{D}_{\Lambda}^{o}$.
(c) $\Lambda(\cdot)$ is steep, namely, $\lim_{n\to\infty}|\nabla\Lambda(\lambda_{n})|=\infty$
whenever $(\lambda_{n})$ is a sequence in $\mathcal{D}_{\Lambda}^{o}$
converging to a boundary point of $\mathcal{D}_{\Lambda}^{o}$.
\end{defn}

Theorem \ref{thm:GE} implies the following statement \cite[Theorem 2.3.6]{Dembo}.
\\
(c) If $\Lambda$ is an essentially smooth, lower semicontinuous function,
then the LDP holds with the good rate function $\Lambda^{*}(\cdot)$.

\begin{proof}
Proof of (a): For any measure $\nu_{n}$, $\nu_{n}(B_{\epsilon}(x))=1$
implies $\mathbb{E}_{\nu_{n}}[X]\in B_{\epsilon}(x)$, since a closed
ball is (completely) convex (due to the fact that the norm $\|\cdot\|_{p}$
with $p\ge1$ is convex). So, by \eqref{eq:-6},
\begin{align}
\gamma_{-}(x) & \ge\lim_{\epsilon\downarrow0}\liminf_{n\to\infty}\inf_{y\in B_{\epsilon}(x)}\gamma_{n}(y)\nonumber \\
 & =\lim_{\epsilon\downarrow0}\liminf_{n\to\infty}\inf_{y\in B_{\epsilon}(x)}\sup_{\lambda\in\mathbb{R}^{d}}\left\langle \lambda,y\right\rangle -\frac{1}{n}\Lambda_{n}(n\lambda)\nonumber \\
 & \ge\sup_{\lambda\in\mathbb{R}^{d}}\lim_{\epsilon\downarrow0}\inf_{y\in B_{\epsilon}(x)}\left\langle \lambda,y\right\rangle -\limsup_{n\to\infty}\frac{1}{n}\Lambda_{n}(n\lambda)\nonumber \\
 & =\sup_{\lambda\in\mathbb{R}^{d}}\left\langle \lambda,x\right\rangle -\Lambda(\lambda)\label{eq:-8}\\
 & =\Lambda^{*}(x),\nonumber 
\end{align}
where \eqref{eq:-8} follows since the linear functional $x\mapsto\left\langle \lambda,x\right\rangle $
is continuous.

The proof of (a) is complete by showing that $\{\mu_{n}\}$ is an
exponentially tight sequence of probability measures. 
\begin{lem}[Exponential Tightness]
\label{lem:Exponential Tightness} \cite[pp. 48-49]{Dembo}  Under
that assumption that the origin is in $\mathcal{D}_{\Lambda}^{o}$,
it holds that 
\begin{equation}
\lim_{\rho\to\infty}\liminf_{n\to\infty}-\frac{1}{n}\log\mu_{n}(\mathbb{R}^{d}\backslash[-\rho,\rho]^{d})=+\infty.\label{eq:-94-1-1-1}
\end{equation}
\end{lem}

Proof of (b):  Fix $y\in\mathcal{F}\cap G$ and let $\eta\in\mathcal{D}_{\Lambda}^{o}$
denote an exposing hyperplane for $y$. Then, for all $n$ large enough,
$\Lambda_{n}(n\eta)<\infty$ and the associated probability measures
$\tilde{\mu}_{n}$ are well-defined via 
\[
\frac{\d\tilde{\mu}_{n}}{\d\mu_{n}}(z)=e^{\left\langle n\eta,z\right\rangle -\Lambda_{n}(n\eta)}.
\]
\begin{lem}[Concentration on Balls]
\cite[pp. 49-50]{Dembo}  \label{lem:concentration} For any $\epsilon>0$,
\begin{align}
\liminf_{n\to\infty}-\frac{1}{n}\log\tilde{\mu}_{n}(B_{\epsilon}(y)^{c}) & >0.\label{eq:-95-1-1-1-1}
\end{align}
In other words, $\tilde{\mu}_{n}(B_{\epsilon}(y)^{c})\to0$ exponentially
fast as $n\to\infty$. 
\end{lem}

Given any $\tilde{\mu}_{n}$, denote an auxiliary probability measure
$\pi_{n}:=\tilde{\mu}_{n}(\cdot|B_{\epsilon}(y))$. Denote $p_{n}:=\tilde{\mu}_{n}(B_{\epsilon}(y))$,
which converges to $1$ by Lemma \ref{lem:concentration}. So, 
\[
\frac{\d\pi_{n}}{\d\mu_{n}}(z)=\frac{e^{\left\langle n\eta,z\right\rangle -\Lambda_{n}(n\eta)}\bone_{B_{\epsilon}(y)}(z)}{p_{n}}.
\]

Since $\pi_{n}$ is a feasible solution to the infimization in \eqref{eq:-6},
\begin{align*}
\limsup_{n\to\infty}-\frac{1}{n}\log\mu_{n}(B_{\epsilon}(y)) & \le\limsup_{n\to\infty}\frac{D(\pi_{n}\|\mu_{n})}{n}\\
 & =\limsup_{n\to\infty}\frac{1}{n}\int\log\frac{e^{\left\langle n\eta,z\right\rangle -\Lambda_{n}(n\eta)}}{p_{n}}\d\pi_{n}(z)\\
 & =\limsup_{n\to\infty}\mathbb{E}_{\pi_{n}}[\left\langle \eta,Z\right\rangle ]-\Lambda(\eta)\\
 & \le\sup_{z\in B_{\epsilon}(y)}\left\langle \eta,z\right\rangle -\Lambda(\eta).
\end{align*}
 So,
\begin{align}
\lim_{\epsilon\downarrow0}\limsup_{n\to\infty}-\frac{1}{n}\log\mu_{n}(B_{\epsilon}(y)) & \le\lim_{\epsilon\downarrow0}\sup_{z\in B_{\epsilon}(y)}\left\langle \eta,z\right\rangle -\Lambda(\eta)\nonumber \\
 & =\left\langle \eta,y\right\rangle -\Lambda(\eta)\label{eq:-7}\\
 & \le\Lambda^{*}(y),\nonumber 
\end{align}
where \eqref{eq:-7} follows since the linear functional $z\mapsto\left\langle \eta,z\right\rangle $
is continuous. 
\end{proof}
It is also possible to generalize the G\"artner--Ellis theorem to
the abstract version, i.e., the Baldi theorem in \cite[Theorem 4.5.20]{Dembo},
by using the entropy method.

\section{\label{sec:Sanov's-Theorems}Sanov's Theorems}

\subsection{\label{sec:Sanov's-theorem}Sanov's Theorem }

The empirical sum is in fact determined by the empirical measure.
We next consider the LD theory of empirical measures and prove Sanov's
theorem by the entropy method. For the finite alphabet case, Sanov's
theorem can be proven by another information-theoretic method, known
as the \emph{method of types}. In fact, Csiszár found that the method
of types, combined with discretization techniques, can be also used
to prove Sanov's theorem for the general alphabet case; see \cite{csiszar2006simple}.
Another proof based on discretization is given in \cite{baldasso2022proof}.
However, our proof given below is based on Csiszár's works on I-projections
in \cite{csiszar1975divergence,csiszar1984sanov}. 

Recall that $\mathcal{X}$ is a Hausdorff topological space (so that
all singletons are closed and hence Borel), and $\mathcal{B}_{\mathcal{X}}$
is the Borel $\sigma$-algebra on $\mathcal{X}$. Recall that $\L_{X^{n}}=\frac{1}{n}\sum_{i=1}^{n}\delta_{X_{i}}$
denotes the empirical measure of $X^{n}\sim P^{\otimes n}$. Consider
the weak topology on $\mathcal{P}(\mathcal{X})$. Let $\mathcal{B}^{w}$
denote the Borel $\sigma$-algebra generated by the weak topology.
Sanov's theorem says that under $\mathcal{B}^{w}$, $\L_{X^{n}}$
satisfies the large deviation principle. 
\begin{defn}
For a set of probability measures $\mathcal{A}\subseteq\mathcal{P}(\mathcal{X})$,
the completely convex hull of $\mathcal{A}$, denoted by $\cconv\mathcal{A}$,
is the set of probability measures $Q_{X}$ such that $Q_{X}=Q_{Z}\circ Q_{X|Z}$
for some probability space $(\mathcal{Z},\Sigma_{\mathcal{Z}},Q_{Z})$
and Markov kernel $Q_{X|Z}$ from $(\mathcal{Z},\Sigma_{\mathcal{Z}})$
to $(\mathcal{X},\mathcal{B}_{\mathcal{X}})$ satisfying $Q_{X|Z=z}\in\mathcal{A}$
for each $z\in\mathcal{Z}$. 
\end{defn}

\begin{defn}
\cite{csiszar1984sanov} A set of probability measures $\mathcal{A}\subseteq\mathcal{P}(\mathcal{X})$
is completely convex if $\mathcal{A}=\cconv\mathcal{A}$. 
\end{defn}

Note that a completely convex set is obviously convex, but the converse
is not true. In fact, the definition above reduces to the one of ``convex''
if we restrict $\mathcal{Z}$ to be finite. An example of completely
convex sets is sets specified by linear constraints, e.g., $\{Q:Q(f)=\alpha\}$
and $\{Q:Q(f)\ge\alpha\}$ for any measurable $f$ and real number
$\alpha$. Another example is closed balls in $\mathcal{P}(\mathcal{X})$
under the Lévy--Prokhorov metric. 
\begin{lem}
\label{lem:Any-closed-ball}Any closed ball $B_{\epsilon]}(P_{X})$
in $\mathcal{P}(\mathcal{X})$ under the Lévy--Prokhorov metric is
completely convex. Moreover, if $\mathcal{X}$ is Polish, then the
completely convex hull of a convex set $\Gamma\subseteq\mathcal{P}(\mathcal{X})$
satisfies that $\Gamma\subseteq\cconv\Gamma\subseteq\overline{\Gamma}$
where $\overline{\Gamma}$ is the closure of $\Gamma$ under the Lévy--Prokhorov
metric. 
\end{lem}

Lemma \ref{lem:Any-closed-ball} implies that if $\mathcal{X}$ is
Polish, then for a (not necessarily convex) set $\Gamma\subseteq\mathcal{P}(\mathcal{X})$,
$\overline{\cconv\Gamma}=\overline{\conv\Gamma}$ where $\conv\Gamma$
is the convex hull of $\Gamma$ under the Lévy--Prokhorov metric.
In other words, under the weak topology, the closed completely convex
hull of a set in $\mathcal{P}(\mathcal{X})$ is just its closed convex
hull. To prove Lemma \ref{lem:Any-closed-ball}, we first observe
that the closure of a set can be written as the intersection of all
its enlargements. 

\begin{lem}
\label{lem:For-a-set}For a set $\Gamma\subseteq\mathcal{P}(\mathcal{X})$,
$\overline{\Gamma}=\inf_{\epsilon>0}\Gamma_{\epsilon]}$, where $\Gamma_{\epsilon]}:=\bigcup_{P_{X}\in\Gamma}B_{\epsilon]}(P_{X})$
for $\epsilon>0$ are closed enlargements of $\Gamma$ under the Lévy--Prokhorov
metric. 
\end{lem}

\begin{proof}
First, $\overline{\Gamma}\subseteq\Gamma_{\epsilon]}$ for all $\epsilon>0$.
Then, $\overline{\Gamma}\subseteq\inf_{\epsilon>0}\Gamma_{\epsilon]}$.
On the other hand, if a point $R\in\Gamma_{\epsilon]}$ for all $\epsilon>0$,
then there is a sequence $\left\{ R_{i}\right\} \subseteq\Gamma$
such that $\Levy(R,R_{i})\to0$, i.e., $R\in\Gamma$ or $R$ is a
limit point of $\Gamma$. Hence, $\overline{\Gamma}\supseteq\inf_{\epsilon>0}\Gamma_{\epsilon]}$.
\end{proof}
We now use Lemma \ref{lem:For-a-set} to prove Lemma \ref{lem:Any-closed-ball}.
\begin{proof}[Proof of Lemma \ref{lem:Any-closed-ball}]
\textbf{}If $Q_{X|Z=z}\in B_{\epsilon]}(P_{X}),\forall z$, then
for any $\delta>\epsilon$,
\begin{equation}
Q_{X|Z=z}(A)\le P_{X}(A_{\delta})+\delta,\forall\textrm{ closed }A\subseteq\mathcal{X},\forall z,\label{eq:-9}
\end{equation}
which, by taking expectations w.r.t. $Q_{Z}$, further implies 
\[
Q_{X}(A)\le P_{X}(A_{\delta})+\delta,\forall\textrm{ closed }A\subseteq\mathcal{X},
\]
i.e., $\Levy(Q_{X},P_{X})\le\delta$. Since $\delta>\epsilon$ is
arbitrary, we have $\Levy(Q_{X},P_{X})\le\epsilon$, i.e., $Q_{X}\in B_{\epsilon]}(P_{X})$.

To prove the second statement, by Lemma \ref{lem:For-a-set}, it suffices
to prove $\cconv\Gamma\subseteq\Gamma_{\epsilon]}$ for all $\epsilon>0$.
 We first prove that $\Gamma_{\epsilon]}$ is convex. If $Q_{X|Z=z}\in\Gamma_{\epsilon]}$,
then there is $P_{X}^{(z)}\in\Gamma$ such that $\Levy(Q_{X|Z=z},P_{X}^{(z)})\le\epsilon$,
which implies for any $\delta>\epsilon$,
\begin{equation}
Q_{X|Z=z}(A)\le P_{X}^{(z)}(A_{\delta})+\delta,\forall\textrm{ closed }A\subseteq\mathcal{X}.\label{eq:-9-1}
\end{equation}
Taking expectations for \eqref{eq:-9-1} w.r.t. discrete distribution
$Q_{Z}$ with finite support, we have for any $\delta>\epsilon$,
\[
Q_{X}(A)\le P_{X}(A_{\delta})+\delta,\forall\textrm{ closed }A\subseteq\mathcal{X},
\]
where $P_{X}:=\sum_{z}Q_{Z}(z)P_{X}^{(z)}$. Hence, $\Levy(Q_{X},P_{X})\le\epsilon$.
By convexity of $\Gamma$, we have $P_{X}\in\Gamma$. So, $Q_{X}\in\Gamma_{\epsilon]}$,
i.e., $\Gamma_{\epsilon]}$ is convex. 

If $\mathcal{X}$ is Polish, then $\mathcal{P}(\mathcal{X})$ with
the Lévy--Prokhorov metric forms a Polish metric space. Let $\hat{\mathcal{P}}$
be a countable dense subset $\mathcal{P}(\mathcal{X})$. Denote $\hat{\Gamma}:=\Gamma_{\epsilon]}\cap\hat{\mathcal{P}}$.
Then, $\hat{\Gamma}=\{R_{i}\}_{i\ge1}$ is a countable dense subset
of $\Gamma_{\epsilon]}$. Denote $\hat{\Gamma}_{\delta]}$ as an closed
enlargement of $\hat{\Gamma}$. Then, $\Gamma_{\epsilon]}\subseteq$$\hat{\Gamma}_{\delta]}\subseteq\Gamma_{\epsilon+\delta]}$.
 Write $\hat{\Gamma}_{\delta]}$ as a partition $\{B_{j}\}$ of $\hat{\Gamma}_{\delta]}$
given by $B_{i}:=B_{\delta]}(R_{i})\backslash\cup_{j=1}^{i-1}B_{j},i\ge1$.
Denote $\mathcal{Z}_{i}:=\{z:Q_{X|Z=z}\in B_{i}\}$ and $q_{i}:=Q_{Z}(\mathcal{Z}_{i})$.
Since $\mathcal{Z}_{i}$ is not necessarily measurable in $\mathcal{Z}$,
to make it measurable, we consider a finer $\sigma$-algebra which
is generated by the sets $\mathcal{Z}_{i}$ and the sets in the $\sigma$-algebra
of $\mathcal{Z}$. Then, under the new $\sigma$-algebra, denote $Q_{Z_{i}}:=Q_{Z}(\cdot|\mathcal{Z}_{i})\circ Q_{X|Z}$
and $Q_{X}:=Q_{Z}\circ Q_{X|Z}=\sum_{i\ge1}q_{i}Q_{Z_{i}}$. Denote
$R_{X}:=\sum_{i\ge1}q_{i}R_{i}$ and $R_{X}^{(n)}:=\frac{\sum_{i=1}^{n}q_{i}R_{i}}{\sum_{i=1}^{n}q_{i}}$. 

Since $B_{i}$ is a subset of $B_{\epsilon]}(R_{i})$, by the first
statement of this lemma, $\Levy(Q_{Z_{i}},R_{i})\le\delta$, which
implies that $\Levy(Q_{X},R_{X})\le\delta$. Moreover, by the definition
of the Lévy--Prokhorov metric, $\Levy(R_{X},R_{X}^{(n)})\le\sum_{i\ge n+1}q_{i}$,
which vanishes as $n\to\infty$. So, $\Levy(Q_{X},R_{X}^{(n)})\le\delta+\sum_{i\ge n+1}q_{i}$.
Since $R_{i}\in\Gamma_{\epsilon]}$ and $R_{X}^{(n)}$ is a convex
combination of $R_{i}$'s, by the convexity of $\Gamma_{\epsilon]}$,
we have $R_{X}^{(n)}\in\Gamma_{\epsilon]}$. So, $Q_{X}\in\Gamma_{\epsilon+2\delta]}$.
Finally, we obtain $\cconv\Gamma\subseteq\cconv\hat{\Gamma}_{\delta]}\subseteq\Gamma_{\epsilon+2\delta]}$.
Since $\epsilon,\delta>0$ are arbitrary, $\cconv\Gamma\subseteq\inf_{\epsilon>0}\Gamma_{\epsilon]}=\overline{\Gamma}$.
\end{proof}
For (completely) convex Borel sets, the following version of Sanov's
theorem holds.
\begin{thm}[Sanov's Theorem for (Completely) Convex Sets]
\label{thm:SanovBounds} The following hold.
\begin{enumerate}
\item \cite[Theorem 1]{csiszar1984sanov} Let $\Gamma\subseteq\mathcal{P}(\mathcal{X})$
be a completely convex Borel set. Then, it holds that 
\begin{align}
\mathbb{P}\{\L_{X^{n}}\in\Gamma\} & \le e^{-nD(\Gamma\|P)}.\label{eq:sanov}
\end{align}
\item Moreover, if $\mathcal{X}$ is Polish, then \eqref{eq:sanov} also
holds for closed convex sets $\Gamma$. 
\end{enumerate}
\end{thm}

The first statement of the theorem above and the proof below are
due to Csiszár \cite[Theorem 1]{csiszar1984sanov}. However, the essentially
same technique was previously used by Massey \cite{Massey1974OnTF},
whose result in fact yields the finite alphabet version of the theorem
above. 
\begin{proof}
Proof of Statement 1: Denote $A=\L^{-1}(\Gamma)$, and $Q_{X^{n}}=P^{\otimes n}(\cdot|A)$.
Then, we have 
\[
-\log P^{\otimes n}(A)=D(Q_{X^{n}}\|P^{\otimes n}).
\]
By the chain rule and the fact that conditioning increases the relative
entropy, the RHS is lower bounded by 
\begin{align}
 & nD(Q_{X_{J}|X^{J-1}J}\|P|Q_{X^{J-1}J})\ge nD(Q_{X_{J}}\|P),\label{eq:-13-1}
\end{align}
where $J\sim\Unif[n]$ is a random time-index independent of $X^{n}$.

On the other hand, $\L_{X^{n}}\in\Gamma$ holds $Q_{X^{n}}$-a.s.
We immediately obtain $\mathbb{E}_{Q_{X^{n}}}\left[\L_{X^{n}}\right]\in\Gamma,$
which can be rewritten as $Q_{X_{J}}\in\Gamma.$

Relaxing $Q_{X_{J}}$ to an arbitrary distribution in $\Gamma$, we
obtain that for any $n\ge1$, 
\begin{align*}
-\frac{1}{n}\log P^{\otimes n}(A) & \ge D(\Gamma\|P).
\end{align*}

Proof of Statement 2: By Lemma \ref{lem:Any-closed-ball}, if $\mathcal{X}$
is Polish, then under the weak topology, a closed convex set is also
a closed completely convex set. Then, applying Statement 1, we obtain
Statement 2. 
\end{proof}
In fact, Sanov's theorem can be extended to any Borel sets.
\begin{thm}[Sanov's Theorem]
\label{thm:sanov} \cite{Sanov61,Dembo} Assume $\mathcal{X}$ is
Polish. For the $\sigma$-algebra $\mathcal{B}^{w}$, the empirical
measures $\L_{X^{n}}$ satisfy the LDP in $\mathcal{P}(\mathcal{X})$
equipped with the weak topology with the good, convex rate function
$D(\cdot\|P)$. 
\end{thm}

Before proving Theorem \ref{thm:sanov}, we introduce the fact that
the law of $\L_{X^{n}}$ is exponentially tight.
\begin{lem}[Exponential Tightness]
\label{lem:exponentialtightness} \cite[Lemma 6.2.6]{Dembo} For
each $b\in(0,\infty)$, there is a compact set $\mathcal{K}_{b}\subseteq\mathcal{P}(\mathcal{X})$
such that 
\begin{equation}
\liminf_{n\to\infty}-\frac{1}{n}\log\mathbb{P}\{\L_{X^{n}}\in\mathcal{K}_{b}^{c}\}\ge b.\label{eq:-94}
\end{equation}
In other words, the law of $\L_{X^{n}}$ is exponentially tight.
\end{lem}

For completeness, the proof of this lemma is given below. 
\begin{proof}[Proof of Lemma \ref{lem:exponentialtightness}]
Choose a non-decreasing sequence $\{K_{j}:j\ge1\}$ of compact subsets
of $\mathcal{X}$ so that $P(K_{j})\le(e-1)\cdot e^{-2j}$, and set
$V=\sum_{j=0}^{\infty}\bone_{\mathcal{X}\backslash K_{j}}$. Then
$V<j$ on $K_{j}$, and so, 
\begin{align*}
P(e^{V}) & =\int_{\bigcup_{j=1}^{\infty}K_{j}}e^{V}\mathrm{d}P=\lim_{j\to\infty}\int_{K_{j}}e^{V}\mathrm{d}P\\
 & =\sum_{j=1}^{\infty}\int_{K_{j}\backslash K_{j-1}}e^{V}\mathrm{d}P\le(e-1)\sum_{j=1}^{\infty}e^{-2j}e^{j}=1.
\end{align*}
At the same time, $V\ge j$ off $K_{j}$, and so for any probability
measure $Q$, $Q(V)\le b\Longrightarrow Q(\mathcal{X}\backslash K_{j})\le\frac{b}{j}$.
By \cite[Theorem 1.12]{prokhorov1956convergence}, $\mathcal{K}_{b}=\{Q:Q(V)\le b\}$
is relatively compact. In addition, because $V$ is lower semi-continuous,
$\mathcal{K}_{b}$ is closed. Hence, for each $b>0$, $\mathcal{K}_{b}$
is compact in $\mathcal{P}(\mathcal{X})$. Finally, 
\begin{align*}
\mathbb{P}\{\L_{X^{n}}\in\mathcal{K}_{b}^{c}\} & =\mathbb{P}\{\L_{X^{n}}(V)>b\}\\
 & \le e^{-nb}\mathbb{E}_{P}[e^{n\L_{X^{n}}(V)}]\\
 & =e^{-nb}\mathbb{E}_{P}[e^{V(X)}]^{n}\\
 & =e^{-nb}.
\end{align*}
\end{proof}
\begin{proof}[Proof of Theorem \ref{thm:sanov}]
 By the general principle given in Theorem \ref{thm:general} and
the exponential tightness given in Lemma \ref{lem:exponentialtightness},
to prove Theorem \ref{thm:sanov}, it suffices to show that $\gamma_{-}(Q)=\gamma_{+}(Q)=D(Q\|P)$
for all $Q\in\mathcal{P}(\mathcal{X})$, where 
\begin{align*}
\gamma_{-}(Q) & =\lim_{\epsilon\downarrow0}\liminf_{n\to\infty}E_{n,\epsilon}(Q)\\
\gamma_{+}(Q) & =\lim_{\epsilon\downarrow0}\limsup_{n\to\infty}E_{n,\epsilon}(Q),
\end{align*}
and 
\[
E_{n,\epsilon}(Q)=-\frac{1}{n}\log\mathbb{P}\{\L_{X^{n}}\in B_{\epsilon}(Q)\}=\inf_{R^{(n)}:R^{(n)}(\L^{-1}(B_{\epsilon}(Q)))=1}\frac{1}{n}D(R^{(n)}\|P^{\otimes n}).
\]

Proof of $\gamma_{-}\ge D(\cdot\|P)$: Since the closed ball $B_{\epsilon]}(Q)$
is completely convex, by Theorem \ref{thm:SanovBounds}, 
\begin{align}
-\frac{1}{n}\log\mathbb{P}\{\L_{X^{n}}\in B_{\epsilon}(Q)\} & \ge D(B_{\epsilon]}(Q)\|P).\label{eq:-95-2}
\end{align}
Taking $\lim_{\epsilon\downarrow0}\liminf_{n\to\infty}$ and by the
lower semicontinuity of $D(\cdot\|P)$, we have $\gamma_{-}(Q)\ge D(Q\|P)$.

Proof of $\gamma_{+}\le D(\cdot\|P)$: Before proving the upper bound
in Theorem \ref{thm:sanov}, we first prove the following lemma. 
\begin{lem}[Concentration on Balls]
 \label{lem:-exponentially-fast-2}  For any $\epsilon>0$, 
\begin{align}
\liminf_{n\to\infty}-\frac{1}{n}\log Q^{\otimes n}(\L^{-1}(B_{\epsilon}(Q)^{c}) & \ge\epsilon^{2}/2.\label{eq:-95-1-2-1}
\end{align}
\end{lem}

This lemma is proven as follows. Since $B_{\epsilon}(Q)^{c}$ is closed,
by Statement (a) of Theorem \ref{thm:general} and the lower bound
proven above, 
\begin{align}
\liminf_{n\to\infty}-\frac{1}{n}\log Q^{\otimes n}(\L^{-1}(B_{\epsilon}(Q)^{c})) & \ge\inf_{R\in B_{\epsilon}(Q)^{c}}\gamma_{-}(R)\ge D(B_{\epsilon}(Q)^{c}\|Q)\ge\epsilon^{2}/2,\label{eq:-95-1-3}
\end{align}
where the last inequality follows by \eqref{eq:D-TV-L}. So, Lemma
\ref{lem:-exponentially-fast-2} holds.

We turn back to proving $\gamma_{+}\le D(\cdot\|P)$. Denote $A_{n}=\L^{-1}(B_{\epsilon}(Q))$.
Denote $p_{n}:=Q^{\otimes n}(A_{n})$, which, by Lemma \ref{lem:-exponentially-fast-2},
converges to $1$ as $n\to\infty$. 

Following steps similar to those in the proof of Theorem \ref{thm:Cramer},
\begin{align*}
\limsup_{n\to\infty}-\frac{1}{n}\log P^{\otimes n}(A_{n}) & \le\limsup_{n\to\infty}\frac{nD(Q_{X}\|P)+H_{2}(p_{n})}{np_{n}}\\
 & =D(Q_{X}\|P).
\end{align*}
Therefore, $\gamma_{+}(Q)\le D(Q\|P)$. 

Combining the two inequalities proven above with the fact that $\gamma_{-}\le\gamma_{+}$,
we obtain $\gamma_{-}(Q)=\gamma_{+}(Q)=D(Q\|P)$ for all $Q\in\mathcal{P}(\mathcal{X})$. 
\end{proof}

\subsection{\label{sec:Sanov's-theorem-1}Strong Sanov's Theorem}

We next strengthen Sanov's theorem for convex sets. The resultant
result is called the \emph{strong Sanov's theorem}. Such a result
incorporates the strong Cramer's theorem as a special case. Recall
that $\mathcal{X}$ is a Hausdorff topological space. 
\begin{thm}[Strong Sanov's Theorem for Convex Sets]
 Let $\Gamma$ be a convex Borel set. Assume that $D(\Gamma\|P)$
is attained by some $Q^{*}$ (i.e., the I-projection exists). Denote
$P_{\iota}$ as the law of $\imath_{Q^{*}\|P}(X)$ with $X\sim Q^{*}$.
Suppose that $P_{\iota}$ is non-lattice, or it is lattice with maximal
step $d$ such that $0<P_{\iota}(D(Q^{*}\|P))<1$. Denote $c=1$ if
$P_{\iota}$ is non-lattice, and $c=\frac{d}{1-e^{-d}}$ if $P_{\iota}$
is lattice. 
\begin{enumerate}
\item It holds that 
\begin{align}
\mathbb{P}\{\L_{X^{n}}\in\Gamma\} & \leq e^{-nD(Q^{*}\|P)},\label{eq:-21-1}
\end{align}
and 
\begin{align}
\mathbb{P}\{\L_{X^{n}}\in\Gamma\} & \lesssim\frac{c}{\sqrt{2\pi n\mathrm{V}(Q^{*}\|P)}}e^{-nD(Q^{*}\|P)}.\label{eq:-21}
\end{align}
\item Moreover, if additionally, $B_{\epsilon}(Q^{*})\cap\mathcal{H}_{+}\subseteq\Gamma$,
then 
\begin{align}
\mathbb{P}\{\L_{X^{n}}\in\Gamma\} & \sim\frac{c}{\sqrt{2\pi n\mathrm{V}(Q^{*}\|P)}}e^{-nD(Q^{*}\|P)},\label{eq:-21-2}
\end{align}
where 
\[
\mathcal{H}_{+}:=\left\{ R:\int\left(\log\frac{\mathrm{d}Q^{*}}{\mathrm{d}P}\right)\mathrm{d}R\ge D(Q^{*}\|P)\right\} .
\]
\end{enumerate}
\end{thm}

In fact, \eqref{eq:-21-1} is restatement of Statement 2 of Theorem
\ref{thm:SanovBounds}. 
\begin{proof}
We first prove \eqref{eq:-21}. Then, similarly to \eqref{eq:-15},
we have 
\begin{align}
\mathbb{P}\{\L_{X^{n}}\in\Gamma\} & =e^{-nD(Q^{*}\|P)}\mathbb{E}\left[\bone_{\{\L_{X^{n}}\in\Gamma\}}e^{-\sum_{i=1}^{n}W_{i}}\right],\label{eq:-15-1}
\end{align}
where $X^{n}\sim Q^{*\otimes n}$ and $W_{i}=\log\frac{\mathrm{d}Q^{*}}{\mathrm{d}P}(X_{i})-D(Q^{*}\|P)$.
Note that $\mathbb{E}\left[W_{i}\right]=0$ and $\mathrm{Var}\left[W_{i}\right]=\mathrm{V}(Q^{*}\|P)$.

A completely convex set must be convex. For convex $\Gamma$ and any
$R\in\Gamma$, it holds that 
\[
D(R\|P)\ge D(R\|Q^{*})+D(Q^{*}\|P).
\]
For $x^{n}$ such that $\L_{x^{n}}\in\Gamma$, substituting $\L_{x^{n}}$
into the inequality above, we have 
\begin{equation}
\mathbb{E}_{\L_{x^{n}}}\left[\log\frac{\mathrm{d}Q^{*}}{\mathrm{d}P}(X)\right]\ge D(Q^{*}\|P).\label{eq:-18}
\end{equation}
The LHS is $\frac{1}{n}\sum_{i=1}^{n}\log\frac{\mathrm{d}Q^{*}}{\mathrm{d}P}(x_{i})$.
Hence, under i.i.d. $X_{i}\sim Q^{*}$, \eqref{eq:-18} reduces to
$\sum_{i=1}^{n}W_{i}\ge0.$ This implies that 
\begin{align*}
\mathbb{E}\left[\bone_{\{\L_{X^{n}}\in\Gamma\}}e^{-\sum_{i=1}^{n}W_{i}}\right] & \le\mathbb{E}\left[\bone_{\{\sum_{i=1}^{n}W_{i}\ge0\}}e^{-\sum_{i=1}^{n}W_{i}}\right]\le1.
\end{align*}
Moreover, by Lemma \ref{lem:expect}, we obtain \eqref{eq:-21}.

We next prove \eqref{eq:-21-2}. We now lower bound the expectation
in \eqref{eq:-15-1}: 
\begin{align}
 & \mathbb{E}\left[\bone_{\{\L_{X^{n}}\in\Gamma\}}e^{-\sum_{i=1}^{n}W_{i}}\right]\nonumber \\
 & \ge\mathbb{E}\left[\bone_{\{\L_{X^{n}}\in B_{\epsilon}(Q^{*})\cap\mathcal{H}_{+}\}}e^{-\sum_{i=1}^{n}W_{i}}\right]\nonumber \\
 & \ge\mathbb{E}\left[\bone_{\{\L_{X^{n}}\in\mathcal{H}_{+}\}}e^{-\sum_{i=1}^{n}W_{i}}\right]\nonumber \\
 & \qquad-\mathbb{E}\left[\bone_{\{\L_{X^{n}}\in\mathcal{H}_{+}\cap B_{\epsilon}(Q^{*})^{c}\}}e^{-\sum_{i=1}^{n}W_{i}}\right].\label{eq:-99}
\end{align}
Note that, as shown above, $\L_{X^{n}}\in\mathcal{H}_{+}$ is equivalent
to $\sum_{i=1}^{n}W_{i}\ge0$. So, the first expectation in the last
line above satisfies 
\begin{align*}
 & \mathbb{E}\left[\bone_{\{\L_{X^{n}}\in\mathcal{H}_{+}\}}e^{-\sum_{i=1}^{n}W_{i}}\right]\\
 & =\mathbb{E}\left[\bone_{\{\sum_{i=1}^{n}W_{i}\ge0\}}e^{-\sum_{i=1}^{n}W_{i}}\right]\\
 & \sim\frac{c}{\sqrt{2\pi n\mathrm{V}(Q^{*}\|P)}},\textrm{ as }n\to\infty.
\end{align*}
The second expectation in the last line above is upper bounded by
\begin{align*}
 & \mathbb{P}\{\L_{X^{n}}\in\mathcal{H}_{+}\cap B_{\epsilon}(Q^{*})^{c}\}\le\mathbb{P}\{\L_{X^{n}}\in B_{\epsilon}(Q^{*})^{c}\}.
\end{align*}
By Lemma \ref{lem:-exponentially-fast-2}, the RHS decays exponentially
fast as $n\to\infty$.

So, we finally obtain 
\begin{align}
\mathbb{P}\{\L_{X^{n}}\in\Gamma\} & \gtrsim\frac{c}{\sqrt{2\pi n\mathrm{V}(Q^{*}\|P)}}e^{-nD(Q^{*}\|P)}.\label{eq:-21-3}
\end{align}
Combining this with \eqref{eq:-21} yields \eqref{eq:-21-2}.
\end{proof}

\section{\label{sec:Gibbs-conditioning-principle}Gibbs Conditioning Principle}

The large deviation result is closely related to the Gibbs conditioning
principle. The latter concerns an important question on conditional
distributions in statistical mechanics. Let $X_{1},X_{2},...,X_{n}$
be a sequence of i.i.d. random variables with each follows $P$ on
a Polish space $\mathcal{X}$. Denote $\L_{X^{n}}$ as its empirical
measure. Given a set $\Gamma\subseteq\mathcal{P}(\mathcal{X})$ and
a constraint $\L_{X^{n}}\in\Gamma$, what is the conditional law of
$X_{1}$ when $n$ is large? In other words, what are the limit points,
as $n\to\infty$, of the conditional probability measures 
\[
Q_{n}(B)=\mathbb{P}\left\{ X_{1}\in B|\L_{X^{n}}\in\Gamma\right\} ,B\in\mathcal{B}_{\mathcal{X}}?
\]

\begin{thm}[Gibbs's Principle]
\cite[Theorem 1]{csiszar1984sanov} Suppose that $\Gamma$ is a completely
convex Borel set such that $I_{\Gamma}:=D(\Gamma^{o}\|P)=D(\overline{\Gamma}\|P)$.
Then, $\{Q_{n}\}$ converges to the unique $Q^{*}$ attaining $D(\overline{\Gamma}\|P)$
as $n\to\infty$ under the relative entropy (and hence also in the
weak topology), i.e., $\lim_{n\to\infty}D(Q_{n}\|Q^{*})=0.$ 
\end{thm}

This theorem is a consequence of Theorem \ref{thm:SanovBounds}. An
example of $\Gamma$ is $\Gamma=\left\{ Q:\mathbb{E}_{Q}[X]\ge\alpha\right\} $
with $\alpha\ge\mathbb{E}_{P}[X]$. The $Q^{*}$ in this case is 
\[
\frac{\mathrm{d}Q^{*}}{\mathrm{d}P}(x)=\frac{e^{\lambda x}}{\mathbb{E}_{P}\left[e^{\lambda X}\right]}
\]
for some $\lambda\ge0$ such that $\mathbb{E}_{Q^{*}}\left[X\right]=\alpha$. 
\begin{proof}
Denote $A_{n}=\{x^{n}:\L_{x^{n}}\in\Gamma\}$, $Q_{X^{n}}=P^{\otimes n}(\cdot|A_{n})$,
and $J\sim\Unif[n]$ is a random time-index independent of $X^{n}$.
Then, it is easily verified that $Q_{X_{J}}=Q_{X_{1}}=Q_{n}.$ By
the proof of Theorem \ref{thm:SanovBounds}, 
\begin{align*}
D(\Gamma^{o}\|P) & \ge\limsup_{n\to\infty}-\frac{1}{n}\log\mathbb{P}\{\L_{X^{n}}\in\Gamma\}\\
 & \ge\limsup_{n\to\infty}D(Q_{n}\|P)\\
 & \ge\liminf_{n\to\infty}D(Q_{n}\|P)\ge D(\Gamma\|P)\ge D(\overline{\Gamma}\|P).
\end{align*}
So, $\lim_{n\to\infty}D(Q_{n}\|P)=I_{\Gamma}.$

For a convex Borel set $\Gamma$ and any $R\in\Gamma$, it holds that
\[
D(R\|P)\ge D(R\|Q^{*})+D(Q^{*}\|P).
\]
So, 
\[
D(Q_{n}\|P)\ge D(Q_{n}\|Q^{*})+D(Q^{*}\|P).
\]
Taking limit, we obtain 
\[
I_{\Gamma}\ge\lim_{n\to\infty}D(Q_{n}\|Q^{*})+I_{\Gamma}.
\]
Hence, we obtain $\lim_{n\to\infty}D(Q_{n}\|Q^{*})=0.$ 
\end{proof}

\section{\label{sec:Proof-of-Theorem}Appendix: Proof of Theorem \ref{thm:duality-I-projection}}

Denote $\Lambda(\lambda):=\log P(e^{\lambda f})$. Denote its effective
domain $\mathcal{D}_{\Lambda}:=\left\{ \lambda\in\mathbb{R}:\Lambda(\lambda)<+\infty\right\} $.
Then, 
\begin{align*}
 & \gamma^{*}(\alpha)=\sup_{\lambda\in\mathcal{D}_{\Lambda}}\lambda\alpha-\log P(e^{\lambda f}).
\end{align*}
Note that $\Lambda(0)=0$ and hence, $0\in\mathcal{D}_{\Lambda}$.
If $\mathcal{D}_{\Lambda}=\{0\}$, then $\bar{\gamma}^{*}(\alpha)=0$
for any $\alpha$. For this case, $P(f)$ does not exist, i.e., $P(f^{+})=P(f^{-})=+\infty$
where $f^{+}=\max\{f,0\}$ and $f^{-}=-\min\{f,0\}$ \cite[Lemma 2.2.5]{Dembo}.

Denote $R_{0}=P(\cdot|[a_{0},b_{0}])$ and $R_{1}=P(\cdot|[a_{1},b_{1}])$,
where we choose $-a_{0},b_{0},-a_{1},b_{1}$ sufficiently large and
satisfying $a_{0}\le a_{1}\le b_{0}\le b_{1}$, $R_{0}(f)\le\alpha$
and $R_{1}(f)\ge\alpha$. Denote $R_{\theta}=\bar{\theta}R_{0}+\theta R_{1}$
with $\theta\in[0,1]$ chosen such that $R_{\theta}(f)=\alpha$. Here
$\bar{\theta}=1-\theta$. Since $R_{\theta}$ is feasible for the
infimization in $\bar{\gamma}(\alpha)$, we have 
\begin{align*}
\bar{\gamma}(\alpha) & \ge D(R_{\theta}\|P)\\
 & =-P([a_{1},b_{0}])\log\left(P([a_{1},b_{0}])\right)\\
 & \qquad-\bar{\theta}P([a_{0},a_{1}])\log\left(\bar{\theta}P([a_{0},a_{1}])\right)\\
 & \qquad-\theta P([b_{0},b_{1}])\log\left(\theta P([b_{0},b_{1}])\right).
\end{align*}
Note that as $-a_{0},b_{0},-a_{1},b_{1}\to\infty$, it holds that
$P([a_{1},b_{0}])\to1$ and $P([a_{0},a_{1}]),P([b_{0},b_{1}])\to0$.
So, $D(R_{\theta}\|P)\to0$, which implies $\bar{\gamma}(\alpha)=0$.
Hence, $\bar{\gamma}^{*}(\alpha)=\bar{\gamma}(\alpha)$ for the case
$\mathcal{D}_{\Lambda}=\{0\}$.

We next consider the case $\mathcal{D}_{\Lambda}\neq\{0\}$. For this
case, $P(f)$ exists (possibly it is equal to $+\infty$ or $-\infty$).

Denote $\lambda_{\max}:=\sup\mathcal{D}_{\Lambda}$ and $\lambda_{\min}:=\inf\mathcal{D}_{\Lambda}$.
Then, $\mathcal{D}_{\Lambda}$ is an interval with endpoints $\lambda_{\min},\lambda_{\max}$.
Note that $\Lambda$ is convex on $\mathbb{R}$, and continuously
differentiable on $\mathcal{D}_{\Lambda}^{o}$. Moreover, $\Lambda(0)=0$
and $\Lambda'(\lambda)=Q_{\lambda}(f)$ where $\frac{\mathrm{d}Q_{\lambda}}{\mathrm{d}P}=\frac{e^{\lambda f}}{P(e^{\lambda f})}$.

Decompose the integral into two parts: $P(e^{\lambda f})=\int_{\{f>0\}}e^{\lambda f}\mathrm{d}P+\int_{\{f\le0\}}e^{\lambda f}\mathrm{d}P$.
By Lebesgue's dominated convergence theorem, $\lambda\mapsto\int_{\{f\le0\}}e^{\lambda f}\mathrm{d}P$
is continuous for $\lambda\ge0$ since $e^{\lambda f}\le1$ for any
$\lambda\ge0$. By Lebesgue's dominated convergence theorem, $\lambda\mapsto\int_{\{f>0\}}e^{\lambda f}\mathrm{d}P$
is continuous for $0\le\lambda<\lambda_{\max}$, and by the monotone
convergence theorem, it is left-continuous at $\lambda_{\max}$ if
$\lambda_{\max}<+\infty$ (even if $\Lambda(\lambda_{\max})=+\infty$).
So, $\Lambda$ is continuous on $[0,+\infty)$ if $\lambda_{\max}=+\infty$,
and $[0,\lambda_{\max}]$ if $\lambda_{\max}<+\infty$. Similarly,
$\Lambda$ is continuous on $(-\infty,0]$ if $\lambda_{\min}=-\infty$,
and $[\lambda_{\min},0]$ if $\lambda_{\min}>-\infty$. Hence, $\Lambda$
is continuous on $\mathcal{D}_{\Lambda}$. Similarly, $\lambda\mapsto Q_{\lambda}(f)$
is also continuous on $\mathcal{D}_{\Lambda}$.

For any $Q$ such that $Q(f)=\alpha$ and $\lambda\in\mathcal{D}_{\Lambda}$,
\begin{align}
D(Q\|P) & \geq\lambda\left(Q(f)-Q_{\lambda}(f)\right)+D(Q_{\lambda}\|P)\label{eq:-55-1-2-1}\\
 & =\lambda\left(\alpha-Q_{\lambda}(f)\right)+D(Q_{\lambda}\|P)\\
 & =\lambda\alpha-\log P(e^{\lambda f}),
\end{align}
which implies $\gamma(\alpha)\ge\gamma^{*}(\alpha).$

We next prove $\gamma(\alpha)\le\gamma^{*}(\alpha).$ We now divide
the proof into four cases.

Case I: We consider the case in which there exists $\lambda^{*}\in\mathcal{D}_{\Lambda}$
such that $Q_{\lambda^{*}}(f)=\alpha$, where $Q_{\lambda}$ a probability
measure with density 
\[
\frac{\mathrm{d}Q_{\lambda^{*}}}{\mathrm{d}P}=\frac{e^{\lambda^{*}f}}{P(e^{\lambda^{*}f})}.
\]
For this case, $Q_{\lambda^{*}}$ is a feasible solution to the infimization
in the definition of $\gamma(\alpha)$. So, 
\[
\gamma(\alpha)\le D(Q_{\lambda^{*}}\|P)=\lambda^{*}\alpha-\log P(e^{\lambda^{*}f})\le\gamma^{*}(\alpha).
\]

We next consider other cases, in which there is no such $\lambda^{*}$
described above, i.e., $Q_{\lambda}(f)>\alpha,\forall\lambda\in\mathcal{D}_{\Lambda}$
or $Q_{\lambda}(f)<\alpha,\forall\lambda\in\mathcal{D}_{\Lambda}$.

Case II: We consider the case in which $\lambda_{\max}=+\infty$ and
$Q_{\lambda}(f)<\alpha,\forall\lambda\in\mathcal{D}_{\Lambda}$, which
implies that $\lim_{\lambda\to\infty}Q_{\lambda}(f)\le\alpha$, i.e.,
$\mathrm{esssup}_{P}f\le\alpha$.

If $\alpha>\mathrm{esssup}_{P}f$, then both $\gamma(\alpha)$ and
$\gamma^{*}(\alpha)$ are equal to $+\infty$.

If $\alpha=\mathrm{esssup}_{P}f$ and $P(f^{-1}(\alpha))>0$, then
define $Q_{\infty}=P(\cdot|f^{-1}(\alpha))$. For this case, $\gamma(\alpha)=D(Q_{\infty}\|P)$
and $\gamma^{*}(\alpha)=\lim_{\lambda\to\infty}-\log P(e^{\lambda(f-\alpha)})$.
Denote $A_{\delta}:=\{f\ge\alpha-\delta\}$ for $\delta>0$ and $A=f^{-1}(\alpha)$.
Then, $P(A)\le P(e^{\lambda(f-\alpha)})\le P(A_{\delta})+e^{-\lambda\delta}$.
Hence, $P(A)\le\lim_{\lambda\to\infty}P(e^{\lambda(f-\alpha)})\le P(A_{\delta})$
for any $\delta>0$, which implies $\lim_{\lambda\to\infty}P(e^{\lambda(f-\alpha)})\le\lim_{\delta\downarrow0}P(A_{\delta})=P(\lim_{\delta\downarrow0}A_{\delta})=P(A_{0})$.
Here, we use the obvious fact $\lim_{\delta\downarrow0}A_{\delta}=A_{0}$.
Moreover, by the definition of essential supremum, $P(A)=P(A_{0})$.
Therefore, $\gamma^{*}(\alpha)=-\log P(A)=D(Q_{\infty}\|P)=\gamma(\alpha)$.

If $\alpha=\mathrm{esssup}_{P}f$ and $P(f^{-1}(\alpha))=0$, then
$P(f\ge\alpha)=0$ which implies that for any $Q\ll P$, $Q(f\ge\alpha)=0$.
Since the probability measure is continuous in events, we have $\lim_{\alpha'\uparrow\alpha}Q(f<\alpha')=Q(f<\alpha)=1$.
So, for any $\delta>0$, there is $\alpha'<\alpha$ such that $1-\delta<Q(f<\alpha')\le1$,
which implies $Q(f)<\alpha.$ Hence, $\gamma(\alpha)=+\infty$. On
the other hand, $\gamma^{*}(\alpha)=\lim_{\lambda\to\infty}-\log P(e^{\lambda(f-\alpha)})$.
Similarly, for any $\delta>0$, there is $\alpha'<\alpha$ such that
$1-\delta<P(f<\alpha')\le1$, which implies for any real $\lambda$,
$P(e^{\lambda f})<e^{\lambda\alpha'}(1-\delta)+e^{\lambda\alpha}\delta.$
Hence, $\lim_{\lambda\to\infty}P(e^{\lambda(f-\alpha)})\le\lim_{\lambda\to\infty}e^{\lambda(\alpha'-\alpha)}(1-\delta)+\delta=\delta$,
which could be arbitrarily small since $\delta>0$ is arbitrary. So,
$\gamma^{*}(\alpha)=+\infty$.

Case III: We now consider the case in which $\lambda_{\max}<+\infty$
and $Q_{\lambda}(f)<\alpha,\forall\lambda\in\mathcal{D}_{\Lambda}$,
which implies that $\lambda_{\max}\in\mathcal{D}_{\Lambda}$ since
otherwise, $\Lambda'(\lambda)=Q_{\lambda}(f)\to\infty$ as $\lambda\uparrow\lambda_{\max}$.
For this case, obviously, 
\begin{align*}
 & \gamma^{*}(\alpha)=\lambda_{\max}\alpha-\log P(e^{\lambda_{\max}f}).
\end{align*}
We next prove that $\gamma(\alpha)$ is also equal to this expression.

For $b>0$, define $f_{b}(x)=\min\{f(x),b\}$. Denote $Q_{b,\lambda}$
as a probability measure with density 
\[
\frac{\mathrm{d}Q_{b,\lambda}}{\mathrm{d}P}=\frac{e^{\lambda f_{b}}}{P(e^{\lambda f_{b}})}.
\]

\begin{fact}
\label{fact:For-all-sufficiently-1} For all sufficiently large $b$,
$Q_{b,\lambda_{\max}}(f)<\alpha$. 
\end{fact}

\begin{fact}
\label{fact:Given-any-,-1} Given any $\lambda>\lambda_{\max}$, it
holds that $Q_{b,\lambda}(f)\to+\infty$ as $b\to\infty$. 
\end{fact}

Fact \ref{fact:For-all-sufficiently-1} follows since by the monotone
convergence theorem, $\lim_{b\to\infty}Q_{b,\lambda_{\max}}(f)=\lim_{b\to\infty}\frac{P(e^{\lambda_{\max}f_{b}}f)}{P(e^{\lambda_{\max}f_{b}})}=\frac{P(e^{\lambda_{\max}f}f)}{P(e^{\lambda_{\max}f})}=Q_{\lambda_{\max}}(f)<\alpha$.

We now prove Fact \ref{fact:Given-any-,-1}. Note that by the monotone
convergence theorem, $P(e^{\lambda f_{b}})\to P(e^{\lambda f})=+\infty$
as $b\to\infty$. Furthermore, for any $b_{0}<b$, 
\begin{align*}
Q_{b,\lambda}(f)\ge Q_{b,\lambda}(f_{b}) & =\frac{P(e^{\lambda f_{b}}f_{b})}{P(e^{\lambda f_{b}})}\\
 & =\frac{\int_{\{f\le b_{0}\}}e^{\lambda f}f\mathrm{d}P+\int_{\{b_{0}<f\le b\}}e^{\lambda f_{b}}f_{b}\mathrm{d}P}{\int_{\{f\le b_{0}\}}e^{\lambda f}\mathrm{d}P+\int_{\{b_{0}<f\le b\}}e^{\lambda f_{b}}\mathrm{d}P}\\
 & \ge\frac{\int_{\{f\le b_{0}\}}e^{\lambda f}f\mathrm{d}P+b_{0}\int_{\{b_{0}<f\le b\}}e^{\lambda f_{b}}\mathrm{d}P}{\int_{\{f\le b_{0}\}}e^{\lambda f}\mathrm{d}P+\int_{\{b_{0}<f\le b\}}e^{\lambda f_{b}}\mathrm{d}P}.
\end{align*}
For fixed $b_{0}$, both the first integrals in numerator and denominator
at the last line are finite, and the integral $\int_{\{b_{0}<f\le b\}}e^{\lambda f_{b}}\mathrm{d}P$
tends to infinity as $b\to\infty$. So, $\liminf_{b\to\infty}Q_{b,\lambda}(f_{b})\ge b_{0}.$
Since $b_{0}$ is arbitrary, $\lim_{b\to\infty}Q_{b,\lambda}(f_{b})=+\infty$,
completing the proof of Fact \ref{fact:Given-any-,-1}.

By Facts \ref{fact:For-all-sufficiently-1} and \ref{fact:Given-any-,-1}
and the continuity of $\lambda\mapsto Q_{b,\lambda}(f)$, for sufficiently
large $b$, there is $\lambda_{b}>\lambda_{\max}$ such that $Q_{b,\lambda_{b}}(f)=\alpha$.
So, 
\begin{equation}
\gamma(\alpha)\le\lambda_{b}\alpha-\log P(e^{\lambda_{b}f_{b}}).\label{eq:-36-2-1-1}
\end{equation}

Letting $b\to\infty$, it holds that $\lambda_{b}\downarrow\lambda_{\max}$
and $f_{b}\uparrow f$, which implies that the RHS converges to $\lambda_{\max}\alpha-\log P(e^{\lambda_{\max}f})$.
Hence, $\gamma(\alpha)\le\lambda_{\max}\alpha-\log P(e^{\lambda_{\max}f})$.

Case IV: We consider the case $Q_{\lambda}(f)>\alpha,\forall\lambda\in\mathcal{D}_{\Lambda}$.
Following arguments similar to those in Cases II and III, $\gamma(\alpha)=\gamma^{*}(\alpha)$
still holds for this case. 

\section*{Acknowledgements}

The author would like to thank Prof. Vincent Y. F. Tan from the National
University of Singapore for his suggestion to investigate the G\"artner--Ellis
theorem by using the entropy method. 

 \bibliographystyle{plainnat}
\bibliography{ref}

\begin{thebibliography}{10}
\providecommand{\natexlab}[1]{#1}
\providecommand{\url}[1]{\texttt{#1}}
\expandafter\ifx\csname urlstyle\endcsname\relax
  \providecommand{\doi}[1]{doi: #1}\else
  \providecommand{\doi}{doi: \begingroup \urlstyle{rm}\Url}\fi

\bibitem[Bahadur and {Ranga Rao}(1980)]{Bah60}
R.~R. Bahadur and R.~{Ranga Rao}.
\newblock On deviations of the sample mean.
\newblock \emph{Annals of Mathematical Statistics}, 31\penalty0 (4):\penalty0
  1015--1027, 1980.

\bibitem[Baldasso et~al.(2022)Baldasso, Oliveira, Pereira, and
  Reis]{baldasso2022proof}
R.~Baldasso, R.~Oliveira, A.~Pereira, and G.~Reis.
\newblock A proof of {Sanov's} theorem via discretizations.
\newblock \emph{Journal of Theoretical Probability}, pages 1--15, 04 2022.
\newblock \doi{10.1007/s10959-022-01174-0}.

\bibitem[Csisz{\'a}r(1975)]{csiszar1975divergence}
I.~Csisz{\'a}r.
\newblock I-divergence geometry of probability distributions and minimization
  problems.
\newblock \emph{The annals of probability}, pages 146--158, 1975.

\bibitem[Csisz{\'a}r(1984)]{csiszar1984sanov}
I.~Csisz{\'a}r.
\newblock {Sanov} property, generalized {$I$}-projection and a conditional
  limit theorem.
\newblock \emph{The Annals of Probability}, 12:\penalty0 768--793, 08 1984.
\newblock \doi{10.1214/aop/1176993227}.

\bibitem[Csisz{\'a}r(2006)]{csiszar2006simple}
I.~Csisz{\'a}r.
\newblock A simple proof of {Sanov's} theorem.
\newblock \emph{Bulletin of the Brazilian Mathematical Society}, 37\penalty0
  (4), 2006.

\bibitem[Dembo and Zeitouni(1998)]{Dembo}
A.~Dembo and O.~Zeitouni.
\newblock \emph{Large Deviations Techniques and Applications}.
\newblock Springer, 2nd edition, 1998.

\bibitem[Gibbs and Su(2002)]{gibbs2002choosing}
A.~L. Gibbs and F.~E. Su.
\newblock On choosing and bounding probability metrics.
\newblock \emph{International statistical review}, 70\penalty0 (3):\penalty0
  419--435, 2002.

\bibitem[Massey(1974)]{Massey1974OnTF}
James~L. Massey.
\newblock On the fractional weight of distinct binary n -tuples (corresp.).
\newblock \emph{IEEE Trans. Inf. Theory}, 20:\penalty0 131, 1974.

\bibitem[Prokhorov(1956)]{prokhorov1956convergence}
Y.~V. Prokhorov.
\newblock Convergence of random processes and limit theorems in probability
  theory.
\newblock \emph{Theory of Probability and Its Applications}, I\penalty0
  (2):\penalty0 157--214, 1956.

\bibitem[Sanov(1961)]{Sanov61}
I.~Sanov.
\newblock On the probability of large deviations of random variables.
\newblock \emph{Mat. Sbornik}, \penalty0 (1):\penalty0 11 -- 44, 1961.

\end{thebibliography}

\end{document}